\documentclass[11pt,draft,reqno]{amsart}
\usepackage{amsmath,amsfonts,amssymb}
\usepackage[utf8]{inputenc}
\usepackage{mathrsfs}
\usepackage{color}
\def\wh{\widehat}
\def\wt{\widetilde}
\def\R{\mathbb R}

\def\Z{\mathbb Z}

\def\N{\mathbb N}
\def\A{\mathbb A}

\def\x{\mathbf x}
\def\y{\mathbf y}
\def\n{\mathbf n}
\def\m{\mathbf m}

\def\k{\mathbf k}
\def\w{\mathbf w}
\def\z{\mathbf z}

\def\D{\mathbf D}

\def\1{\mathbf 1}

\def\H{\mathfrak H}

\def\bxi{\boldsymbol \xi}

\def\1{\bold 1}

\def\eps{\varepsilon}
\def\Dom{\mathrm{Dom}\,}
\def\Ker{\mathrm{Ker}\,}

\def\leq{\leqslant}
\def\le{\leqslant}

\def\ge{\geqslant}

\usepackage[height = 23.5cm, a4paper, hmargin={2.5cm , 2.0cm}]{geometry}

\usepackage{color}

\definecolor{darkred}{rgb}{0.9,0.1,0.1}

\theoremstyle{theorem}
\newtheorem{theorem}{Theorem}[section]
\newtheorem{proposition}[theorem]{Proposition}
\newtheorem{lemma}[theorem]{Lemma}

\newtheorem{remark}[theorem]{Remark}

\numberwithin{equation}{section}

\theoremstyle{plain}
\newtoks\thehProclaim
\newtheorem*{Proclaim}{\the\thehProclaim}

\begin{document}
\thispagestyle{empty}

\bigskip

\centerline{\textbf{Operator estimates in homogenization of L\'evy-type operators}}
\centerline{\textbf{with periodic coefficients}}

\bigskip
\def\supind#1{${}^\mathrm{#1}$}
\centerline{\textbf{A.~Piatnitski\supind{1,2}, V.~Sloushch\supind{3},
T.~Suslina\supind{3}, E.~Zhizhina\supind{1,2}}}

\bigskip
\centerline{\supind{1} The Arctic University of Norway, UiT, campus Narvik,}
\centerline{Lodve Langes gate 2, Narvik 8517, Norway}

\bigskip
\centerline{\supind{2}Higher School of Modern Mathematics MIPT,}
\centerline{1st Klimentovskiy per., 115184  Moscow, Russia }

\bigskip
\centerline{\supind{3}St.~Petersburg State University,}
\centerline{Universitetskaya nab. 7/9, St.~Petersburg, 199034, Russia }

\bigskip

\centerline{e-mail: apiatnitski@gmail.com}
\centerline{e-mail: v.slouzh@spbu.ru}
\centerline{e-mail: t.suslina@spbu.ru}
\centerline{e-mail: elena.jijina@gmail.com}

\bigskip
\bigskip
\noindent
%\begin{abstract}
{\bf Abstract.}\\
The paper deals with homogenization of self-adjoint operators in $L_2(\mathbb R^d)$ of the form
$$
({\mathbb A}_\eps u) (\x) =  \int_{\R^d} \mu(\x/\eps, \y/\eps) \frac{\left( u(\x) - u(\y) \right)}{|\x - \y|^{d+\alpha}}\,d\y,
$$
where $0< \alpha < 2$, and $\eps>0$ is a small  parameter.
It is assumed that the function $\mu(\x,\y)$ is $\Z^d$-periodic in each variable, $\mu(\x,\y)=\mu(\y,\x)$ for all $\x$ and $\y$,
and $0< \mu_- \leqslant \mu(\x,\y) \leqslant \mu_+< \infty$.
Under these assumptions we show that the resolvent $({\mathbb A}_\eps + I)^{-1}$ converges, as $\eps\to0$, in the operator
norm in $L_2(\R^d)$ to the resolvent $({\mathbb A}^0 + I)^{-1}$ of the limit operator ${\mathbb A}^0$
given by
$$
({\mathbb A}^0 u) (\x) =  \int_{\R^d} \mu^0 \frac{\left( u(\x) - u(\y) \right)}{|\x - \y|^{d+\alpha}}\,d\y,
$$
where $\mu^0$ is the mean value of  $\mu(\x,\y)$.
We also show that the operator norm of the discrepancy $\|({\mathbb A}_\eps + I)^{-1} - (\A^0 + I)^{-1}\|_{L_2(\mathbb R^d)\to L_2(\mathbb R^d)}$ can be estimated  by $O(\eps^\alpha)$, if  $0< \alpha < 1$, by $O(\eps (1 + | \operatorname{ln} \eps|)^2)$,
if $ \alpha =1$, and by $O(\eps^{2- \alpha})$,  if $1< \alpha < 2$.
%\end{abstract}

\bigskip
\noindent\textbf{Keywords}:
L\'evy-type operators, periodic homogenization, operator estimates of discrepancy, effective operator.

%{\color{cyan}
%}

\section*{Introduction}

In this work we consider homogenization problems for operators of L\'evy type with periodic coefficients.
Our main goal is to obtain estimates for the rate of convergence in the operator norms.

\subsection{Problem setup.  Main result}
This work deals with unbounded nonlocal L\'evy-type operators  $\A_\eps = \A_\eps(\alpha,\mu)$ in $L_2(\R^d)$ that
are formally defined by
 \begin{equation}
 \label{Aeps_Intro}
 (\mathbb{A}_\eps u)( \x)= \int\limits_{\mathbb R^d}  \mu ( \x/\eps,
  \y/\eps)\frac{( u(\x)-u(\y))}{|\x - \y|^{d+\alpha}}\,d\y,\ \ \x\in\mathbb{R}^{d};
 \end{equation}
here $0< \alpha <2$,  $\eps$ is a small positive parameter, and $\mu(\x,\y)$ is
bounded positive definite function, which is
$\mathbb{Z}^{d}$-periodic both in $\x$ and $\y$. Moreover, we assume that $\mu(\x,\y)=\mu(\y,\x)$.
In order to introduce  $\A_\eps$ in a rigorous way we consider the closed quadratic form
\begin{equation*}
\label{q_form_Intro}
a_\eps [u,u] := \frac{1}{2} \intop_{\R^d} \intop_{\R^d} d\x\,d\y\, \mu(\x/\eps,\y/\eps) \frac{|u(\x)-u(\y)|^2}{|\x - \y|^{d+\alpha}},\ \ u\in H^{\frac\alpha2}(\R^d),
\end{equation*}
and define $\A_\eps$ as the non-negative self-adjoint operator generated by this form.

%{\color{magenta}
Notice that the quadratic form $a_\eps [u,u]$ with the domain
$H^{\frac{\alpha}{2}}(\mathbb R^d)$ is a regular Dirichlet form.
The Markov process (Hunt process) that corresponds to this form
is a Markov  jump process, its infinitesimal generator coincides with
$-\A_\eps$. A detailed description of the properties of such an operator
can be found in the work  \cite{BBZK}. Since in this paper the probabilistic interpretation
of solutions is not used, we do not  provide this description here.
The kernel of the operator $\A_\eps$ exhibits a power-law decay at infinity, its second moment is not finite.
An important characteristic of the trajectories of such processes is the presence of long-distance jumps (L\'evy flights).
This makes an essential difference with the behaviour of trajectories of diffusion processes.
Since the quadratic form $a_\eps [u,u]$ is comparable with the quadratic form of the fractional Laplacian
$(-\Delta)^\frac\alpha2 $, then, abusing slightly the terminology, we call a process that corresponds
to the form $a_\eps [u,u]$ a L\'evy-type process, and $\A_\eps$ -- a L\'evy-type operator.
Currently, L\'evy processes are widely used in describing the behaviour of complex systems in which
long distance interactions play an important and sometimes a key role.
In particular, various models in population biology and ecology, in astrophysics,  mechanics of porous media
and financial mathematics are based on these processes, see, for instance,   \cite{CT, EP, HS, NS, UZ, W}.
When studying similar models in heterogeneous media, we arrive  at the processes with a generator $-\A_\eps$
defined in \eqref{Aeps_Intro}.
%}
%{\color{blue}
%}

Homogenization problem for the operator $\A_\eps$  was addressed in \cite{KaPiaZhi19},  where it was shown that,
as $\eps\to0$, the resolvent  $(\A_\eps + I)^{-1}$  converges strongly in $L_2(\mathbb R^d)$ to the resolvent
$(\A^0 + I)^{-1}$ of the effective operator $\A^0$. This operator, being an operator of the same form as $\A_\eps$,
has a constant coefficient
$$
\mu^0 = \intop_\Omega \intop_{\Omega} \mu(\x,\y)\,d\x\,d\y.
$$
Notice that such an operator coincides with the fractional Laplacian up to a multiplicative constant:
$\A^0 = \mu^0 c_0(d,\alpha) (-\Delta)^{\frac\alpha2}$, $\Dom \A^0 = H^{\alpha}(\R^d)$.

Our \emph{main result}, Theorem \ref{teor3.1},  states that the resolvent
$(\A_\eps + I)^{-1}$ converges to the resolvent of the effective operator in the operator norm in $L_2(\mathbb R^d)$, moreover, the following estimate for the rate of convergence holds:
\begin{equation}\label{e3.1_Intro}
\|(\A_{\varepsilon}+I)^{-1}-(\A^{0}+I)^{-1}\|_{L_2(\R^d) \to L_2(\R^d)}\le
{\mathrm C}(\alpha,\mu) \begin{cases} \eps^\alpha, &\hbox{if } 0 < \alpha <1,\\
\eps (1 + | \operatorname{ln} \eps|)^2, &\hbox{if } \alpha =1,
\\ \eps^{2 - \alpha}, &\hbox{if } 1< \alpha < 2.
\end{cases}
\end{equation}

\subsection{Homogenization theory methods based on Floquet--Bloch representation. Operator estimates.}\label{Sec0.1}

Currently, homogenization theory is a well-developed area of mathematics that comprises various methods and techniques,
see, for instance, the monographs \cite{BaPa}, \cite{BeLP}, \cite{JKO}.
One of the important approaches in this area relies on the Floquet--Bloch theory
and on the Gelfand transform.    The first rigorous homogenization result obtained by this method can be found in \cite{Sev},
where for a second order uniformly elliptic operator with periodic coefficients the strong resolvent convergence
in $L_2(\mathbb R^d)$ was proved.     Later on this approach was further developed in \cite{AlCo98}, \cite{AlCoVa98}, \cite{COrVa}, \cite{Zh89} and in other papers. The cited works studied various homogenization problems for differential operators with periodic microstructure,
among them are boundary-value problems for elliptic operators, related spectral problems, operators in perforated domains and fluid mechanics problems.
Notice, however, that the mentioned works were aimed at proving strong resolvent convergence.

In \cite{BSu1, BSu3, BSu4}  Birman and Suslina introduced a new approach to the study of homogenization problems
for  differential operators with periodic coefficients. This approach is based on a version of the  spectral method. The advantage of this approach is that it allows to obtain
sharp estimates for the rate of convergence \emph{in the operator norms}  for a wide class
of homogenization problems in periodic environments.
To illustrate this approach we consider in $L_2(\mathbb R^d)$ a scalar uniformly elliptic operator of the form
$\mathcal{A}_\eps=-\mathrm{div} \, g\big(\frac \x\eps\big)\nabla, \, \eps>0$, with  $\mathbb Z^d$-periodic coefficients.
As follows from classical homogenization results,  for such an operator the strong resolvent convergence holds,
as $\eps\to0$, and the limit operator has the form $\mathcal{A}^0 = - \operatorname{div} g^0 \nabla$, where $g^0$ is a constant positive definite \emph{effective matrix}.  See, for instance, \cite{BeLP} for further details.
As was shown in  \cite{BSu1} a more advanced convergence result holds. Namely, the resolvent $(\mathcal{A}_\eps +I)^{-1}$  converges to the resolvent of the limit operator
in the operator norm in $L_2(\R^d)$, and the following estimate is valid:
\begin{equation}
\label{BSu1}
\| ( \mathcal{A}_\eps +I)^{-1} - (\mathcal{A}^0 +I)^{-1}\|_{L_2(\R^d) \to L_2(\R^d)} \leqslant C \eps.
\end{equation}

The estimates of this type are called operator estimates for the rate of convergence in homogenization theory.
In  \cite{BSu3}  a more detailed asymptotics of the resolvent  $(\mathcal{A}_\eps +I)^{-1}$ that includes additional terms with correctors and provides
an approximation precision of order  $O(\eps^2)$ in the operator norm in
$L_2(\mathbb R^d)$ was constructed.
An approximation of the resolvent  $( \mathcal{A}_\eps +I)^{-1}$  in the norm of operators acting from $L_2(\R^d)$ into the
Sobolev space $H^1(\R^d)$ was obtained in \cite{BSu4}.

The operator-theoretic approach is based on the scaling transformation,  Floquet--Bloch theory and  analytic perturbation theory.
To clarify the method, we provide here the derivation of  estimate \eqref{BSu1}.
The scaling transformation reduces estimate  \eqref{BSu1} to the following inequality:
\begin{equation}
\label{BSu2}
\| (\mathcal{A} + \eps^2 I)^{-1} - (\mathcal{A}^0 + \eps^2 I)^{-1}\|_{L_2(\R^d) \to L_2(\R^d)} \leqslant C \eps^{-1},
\end{equation}
where $\mathcal{A} = - \operatorname{div} g(\x) \nabla = \D^* g(\x) \D$, $\D = -i \nabla$.
With the help of the unitary Gelfand transform the operator $\mathcal{A}$ can be expanded into a direct integral
over the family of  operators  $\mathcal{A}(\bxi)$  in $L_2(\Omega)$ with $\bxi \in \widetilde{\Omega}$, $\bxi$ is called quasi-momentum;
here  $\Omega = [0,1)^d$  is the unit cell of the lattice  $\Z^d$, $\widetilde{\Omega} = [-\pi,\pi)^d$ is the cell of the dual lattice, and the operator $\mathcal{A}(\bxi)$  is defined by $\mathcal{A}(\bxi) = (\D + \bxi)^* g(\x) (\D+ \bxi)$ with periodic boundary conditions.
Estimate \eqref{BSu2} follows from similar estimates for the operators depending on quasi-momentum:
\begin{equation*}
\| (\mathcal{A}(\bxi) + \eps^2 I)^{-1} - (\mathcal{A}^0(\bxi) + \eps^2 I)^{-1}\|_{L_2(\Omega) \to L_2(\Omega)} \leqslant C \eps^{-1},
\quad \bxi \in \widetilde{\Omega}.
\end{equation*}
An important part of the study is the analysis of the family of operators  $\mathcal{A}(\bxi)$.
Since $\{\mathcal{A}(\bxi)\}$ is an analytic family of operators with a compact resolvent,
the methods of analytic perturbation theory apply. It turned out that the resolvent
\hbox{$(\mathcal{A}(\bxi) + \eps^2 I )^{-1}$} can be approximated in terms of the spectral characteristics
 of the operator at its spectral edge. Therefore, homogenization represents a  \emph{spectral threshold effect} at the
 spectral edge of the elliptic operator.

Another method of obtaining operator estimates for the rate of convergence in homogenization problems, the so-called
shift method, was developed in the works of Zhikov and Pastukhova, see  \cite{Zh, ZhPas1}, as well as the survey \cite{ZhPas3} and the literature cited there.

Recent years operator estimates for the rate of convergence in various homogenization problems for differential operators
have attracted the attention of an increasing number of researchers. A number of deep and meaningful results have been
obtained in this field. A detailed overview of the current state of this area can be found in  \cite[Introduction]{Su_UMN2023}.

\subsection{Existing homogenization results for convolution-type  operators}
For the first time, operator estimates for the rate of convergence in homogenization problems for
\emph{nonlocal convolution-type periodic operators} were obtained in recent works
of the authors  \cite{PSlSuZh, PSlSuZh2}. These works focused on convolution-type operators
in $L_2(\R^d)$ that read
 \begin{equation}
 \label{Sus1}
 ({A}_\eps u)( \x)=\frac1{\eps^{d+2}}\int\limits_{\mathbb R^d} a((\x-\y)/\eps) \mu ( \x/\eps,
  \y/\eps)\big( u(\x)-u(\y)\big)\,d\y,\ \ \x\in\mathbb{R}^{d},\ \ u\in L_{2}(\mathbb{R}^{d});
 \end{equation}
here $\eps >0$ is a small parameter.
These operators were studied under the assumptions that $a(\x)$ is an even, non-negative integrable function,
and the function  $\mu(\x,\y)$ is bounded, positive definite, symmetric in $\x$ and $\y$, and $\mathbb{Z}^{d}$-periodic
in each variable. Under these conditions the operator  ${A}_{\eps}$ is bounded in $L_2(\mathbb R^d)$, self-adjoint
and non-negative.
In addition, it was assumed that the moments $M_{k}(a) =\int_{\mathbb{R}^{d}}|\x|^{k}a(\x)\,d\x$ are finite up to order
$3$ or $4$.
Convolution-type operators with integrable kernels arise in the models of mathematical biology and population
dynamics, recent years these models were actively studied in the mathematical literature, see \cite{KPMZh, PZh, PiaZhi19}.
Periodic homogenization problem for such operators was considered in the work \cite{PZh}, where, under the assumption that $M_2(a) < \infty$, it was shown that the resolvent $({A}_{\varepsilon}+I)^{-1}$ converges strongly in $L_2 (\mathbb R^d)$ to the resolvent $({A}^{0}+I)^{-1}$ of the effective operator. The effective operator has the form
 ${A}^{0}=-\operatorname{div}g^{0}\nabla$ with a positive definite constant matrix $g^0$.
 It is interesting to observe that in these models, although the original operator ${A}_{\eps}$ is nonlocal
 and bounded, the homogenized one is local and unbounded.

Homogenization problems for convolution-type operators with non-symmetric kernels were studied in   \cite{PiaZhi19},
where, for the corresponding parabolic semigroups, the convergence result was obtained in moving coordinates.
Similar problems in perforated domains were investigated by variational methods in  \cite{BraPia21}.

In the article  \cite{PSlSuZh}, under the assumption that  $M_3(a) < \infty$, the authors proved
the following sharp in order estimate for the rate of convergence:
\begin{equation*}\label{Sloushch2}
\|({A}_{\varepsilon}+I)^{-1}-({A}^{0}+I)^{-1}\|_{L_{2}(\mathbb{R}^{d})\to
L_{2}(\mathbb{R}^{d})}\leqslant C(a,\mu)\varepsilon.
\end{equation*}
Then in \cite{PSlSuZh2}, under the additional condition $M_4(a)<\infty$,  the convergence result was improved by using the correctors.

In   \cite{PSlSuZh,  PSlSuZh2} the operator-theoretic approach developed by M. Birman and T. Suslina in \cite{BSu1} for differential operators was modified and adapted to the case of nonlocal convolution-type operators with integrable kernels.
As in the case of differential operators, the homogenization problem is reduced to studying the operator family $A(\bxi)$,
$\bxi\in\widetilde\Omega$,
obtained from the original operator by the scaling transformation and the Gelfand transform. However, in contrast with
the differential operators, this family is not analytic, and thus the analytic perturbation theory does not apply.
Instead, the authors used the finite smoothness of $A(\bxi)$  granted by the moment condition.
%{\color{cyan}
%}

\subsection{Methods} In order to investigate the limit behaviour of the resolvent of $\A_\eps$ we modify the operator-theoretic approach and adapt it to the case of L\'evy-type operators.

At the first step, making the scaling transformation we obtain
\begin{equation}\label{e3.2_Intro}
\|(\A_{\varepsilon}+I)^{-1}-(\A^{0}+I)^{-1}\|_{L_2(\R^d) \to L_2(\R^d)} = \eps^\alpha
\|(\A + \eps^\alpha I)^{-1}-(\A^{0}+ \eps^\alpha I)^{-1}\|_{L_2(\R^d) \to L_2(\R^d)};
\end{equation}
here $\A = \A_{\eps_0},\ \eps_0 =1$.

Then we apply the Gelfand transform and represent the operator  $\A$ as a direct integral over the family of operators $\A(\bxi)$  in $L_2(\Omega)$ that depend on the parameter $\bxi \in \wt{\Omega}$. For each $\bxi \in \wt{\Omega}$ the spectrum of $\,A(\bxi)$
is discrete and belongs to $\mathbb R_+$; for small $|\bxi|$ the first eigenvalue is of order  $O(|\bxi|^\alpha)$,
while the other eigenvalues are uniformly positive.

Now the question of the limit behaviour of $(\mathbb{A}_\eps+ I)^{-1}$, as $\eps\to0$, is reduced to studying the asymptotics
of the resolvent \hbox{$(\mathbb{A}(\bxi)+\eps^{\alpha}I)^{-1}$} for small $\eps>0$.
It is clear that the main contribution to the asymptotics under consideration  comes from the bottom part of the spectrum of  $\A(\bxi)$. It should be emphasized that, in contrast with the case of differential operators, the family  $\A(\bxi)$
is not analytic, and for  $\alpha\in (0,1)$ not even differentiable.
In this way the studied L\'evy-type operators differ essentially  from the convolution-type operators of the form  \eqref{Sus1},
for which the finite differentiability of  $\A(\bxi)$ is ensured by the finiteness of an appropriate number of moments of
 $a(\x)$.
Nevertheless, we succeeded to obtain the ''threshold approximations'' required for constructing an approximation
of the resolvent  $(\mathbb{A}(\bxi)+\eps^{\alpha}I)^{-1}$ for small $\eps$.
Namely, we obtained approximations of the operators $F(\bxi)$ and $\mathbb{A}(\bxi)F(\bxi)$, as $\bxi\to0$.
Here $F(\bxi)$ is a spectral projector of the operator  $\mathbb{A}(\bxi)$ that corresponds to a neighbourhood of zero.
In the existing literature the asymptotics of the operator $\mathbb{A}(\bxi)F(\bxi)$ for small $\bxi$ was usually
determined in terms of the asymptotics for the principal eigenvalue $\lambda_1(\bxi)$ of $\mathbb{A}(\bxi)$.
Here we apply an alternative approach that relies on integration of the resovent $(\A(\bxi) - \zeta I)^{-1}$
over an appropriate contour in the complex plane.

Since $\| (\A(\bxi) + \eps^\alpha I)^{-1} F(\bxi)^\perp\| \le C$, the best precision that can be obtained when
approximating the resolvent $(\A_{\varepsilon}+I)^{-1}$ by the described above method is
$O(\eps^\alpha)$; see \eqref{e3.2_Intro}. Therefore, for $0< \alpha < 1$ the principal term of the
 expansion provides the optimal precision,  see \eqref{e3.1_Intro}. However, for $1 \le \alpha <2$ it is not the case,
 and the precision can be improved by means of adding  correctors. The authors are planning to address this problem
 in a separate paper.

\subsection{Structure of the paper} The paper consists of Introduction and five sections.
In Section 1 we introduce the operator ${\mathbb A}$, expand this operator into a direct integral over the operator family ${\mathbb A}(\bxi)$ and obtain lower bounds for the quadratic forms of the operators ${\mathbb A}(\bxi)$.
 Section 2 deals with upper bound for the difference between the quadratic forms $a(\bxi)$ and $a(\mathbf{0})$.
The threshold characteristics of the operator family ${\mathbb A}(\bxi)$ near the bottom of the spectrum are studied
in Section 3.
In Section 4 we first approximate the resolvent $( {\mathbb A}(\bxi) + \eps^\alpha I)^{-1}$ for small $\eps$.
Then combining this approximation with the representation of operator $\A$ as a direct integral over the operators
 ${\mathbb A}(\bxi) $, we construct an approximation for the resolvent $( {\mathbb A} + \eps^\alpha I)^{-1}$.
Finally, in Section 5 the main result of this work is derived from the results of Section 4 by means of the scaling
transformation. Namely, we obtain an approximation of the resolvent $({\mathbb A}_\eps + I)^{-1}$ in the operator
norm in $L_2(\R^d)$.

\subsection{Notation}
The norm in a normed space $X$ is denoted by $\|\cdot\|_{X}$, the index is dropped if it does not lead to
misunderstanding.  For normed linear spaces $X$ and $Y$ the standard norm of a linear continuous operator
$T:X\to Y$  is denoted by $\|T\|_{X\to
Y}$ or just $\|T\|$. The notation $\mathcal{L}\{F\}$ is used for the linear span of a collection of vectors  $F\subset X$.

For given complex separable Hilbert spaces  $\H$ and  $\H_*$, and a linear operator   $A: \H \to \H_*$,
the notation $\operatorname{Dom} A$ is used for the domain of $A$, and $\operatorname{Ker} A$ --- for its kernel.
For a self-adjoint operator  $\A$ in a Hilbert space  $\mathfrak{H}$ the symbol $\sigma(\A)$ stands for the
spectrum of $\A$; if ${\tiny{\boldsymbol{\Delta}}}$ is a subset of the real line $\R$, then the spectral projector
of the operator $\A$ that corresponds to the subset ${\tiny{\boldsymbol{\Delta}}}$ is denoted by
$\normalsize{E_{\A}(}$\!${\tiny{\boldsymbol{\Delta}}}$\!$)$.

For a domain $\mathcal O\subset \R^d$ we denote by $L_{p}({\mathcal O})$,
\hbox{$1 \le p \le \infty$}, the standard  $L_p$ spaces. The  inner product
in  $L_{2}({\mathcal O})$ is denoted by  $(\cdot,\cdot)_{L_{2}({\mathcal O})}$ or just
$(\cdot,\cdot)$.
We  use the standard notation $H^s({\mathcal O})$, $s>0$, for the corresponding Sobolev spaces.
Also, in what follows we use the following notation: $i D_j = \partial_j = \partial / \partial x_j$,
$j=1,\dots,d$; $\mathbf{D} = - i \nabla = (D_1,\dots,D_d)^t$;
 $\mathcal{S}(\R^{d})$ stands for the Schwartz class in  $\R^{d}$.
The characteristic function of a set  ${\mathcal O}\subset\R^d$ is denoted  by
$\mathbf{1}_{\mathcal O}$.

\section{L\'evy--type operators with periodic coefficients: \\ expansion into a direct integral and a priori estimates}

\subsection{Operator $\A(\alpha,\mu)$}
We begin by introducing the operator  $\A(\alpha,\mu)$.
Given a function $\mu\in L_{\infty}(\R^d\times \R^d)$ such that
\begin{gather}
\label{e1.1}
0<\mu_{-}\le\mu(\x,\y)\le\mu_{+}< \infty,\ \ \mu(\x,\y)=\mu(\y,\x),\ \ \x, \y\in\R^d;
\\
\label{e1.2}
\mu(\x+\m,\y+\n)=\mu(\x,\y),\ \ \x, \y\in\R^d,\ \ \m,\n\in\Z^d,
\end{gather}
we define in $L_{2}(\R^d)$ the following quadratic form:
\begin{equation}
\label{e1.3}
a(\alpha,\mu) [u,u] := \frac{1}{2} \intop_{\R^d} \intop_{\R^d} d\x\,d\y\, \mu(\x,\y) \frac{|u(\x)-u(\y)|^2}{|\x - \y|^{d+\alpha}},\ \ u\in H^\gamma(\R^d),
\end{equation}
where $0 < \alpha < 2$ and $\gamma := \frac{\alpha}{2}$.
Due to \eqref{e1.1} and \eqref{e1.3} the form $a(\alpha,\mu)$ is non-negative, has a dense domain and
satisfies the estimates
 \begin{equation}
\label{e1.4}
\mu_- a_0(\alpha) [u,u] \le a(\alpha,\mu) [u,u] \le \mu_+ a_0(\alpha) [u,u],\ \ u\in H^\gamma(\R^d);
\end{equation}
here
\begin{equation}
\label{e1.5}
a_0(\alpha) [u,u] := \frac{1}{2} \intop_{\R^d} \intop_{\R^d} d\x\,d\y\, \frac{|u(\x)-u(\y)|^2}{|\x - \y|^{d+\alpha}},\ \ u\in H^\gamma(\R^d).
\end{equation}
The propertiers of the form $a_0(\alpha) $ are well known. A proof of the following statement can be found for instance
in  \cite[\S~6.31]{Sa}. For the reader convenience we provide a sketch of the proof here.
\begin{lemma}
\label{lem1.1}
Assume that  $0< \alpha < 2$. Then the form  $a_0(\alpha)$ admits the following representation:
\begin{equation}
\label{e1.6}
a_0(\alpha) [u,u] = c_0(d,\alpha)  \intop_{\R^d} d\k\, |\k|^\alpha | \wh{u}(\k)|^2,\ \ u\in H^\gamma(\R^d);
\end{equation}
here $\wh{u}(\k)$ is the Fourier image of a function $u(\x)$, and the constant  $c_0(d,\alpha)$
is defined by
\begin{equation}
\label{e1.6a}
 c_0(d,\alpha) =  \intop_{\R^d} \frac{1 - \cos z_1}{ | \z|^{d+\alpha}} \,d\z = \frac{ \pi^{d/2} |\Gamma(-\frac\alpha2)|}{2^\alpha \Gamma((d+\alpha)/2)}.
\end{equation}
\end{lemma}

\begin{proof}
Making the change of variables $\y \mapsto \z = \y-\x$ in the integral on the right-hand side of \eqref{e1.5}
and taking into account the unitarity of Fourier transform we obtain
\begin{equation*}
\begin{aligned}
a_0(\alpha) [u,u] &= \frac{1}{2} \intop_{\R^d} \intop_{\R^d} d\x\,d\z\, \frac{|u(\x)-u(\x+\z)|^2}{|\z|^{d+\alpha}}
\\
&= \frac{1}{2} \intop_{\R^d} \intop_{\R^d} d\k\,d\z\, \frac{| \wh{u}(\k)|^2 |1 -e^{i \langle \k,\z\rangle}|^2}{|\z|^{d+\alpha}}
=  \intop_{\R^d} d\k\,  V_\alpha(\k) | \wh{u}(\k)|^2,
\end{aligned}
\end{equation*}
where
\begin{equation}
\label{e1.6b}
V_\alpha(\k) := \frac{1}{2} \intop_{\R^d} \,d\z\, \frac{|1 -e^{i \langle \k,\z\rangle}|^2}{|\z|^{d+\alpha}} =
 \intop_{\R^d} \,d\z\, \frac{1 - \cos( \langle \k,\z\rangle)}{|\z|^{d+\alpha}}.
\end{equation}
It is easy to check that
\begin{equation}
\label{e1.6c}
V_\alpha(\k) = c_0(d,\alpha) |\k|^\alpha, \quad \k \in \R^d,
\end{equation}
where the constant $c_0(d,\alpha)$ is defined by the first relation in \eqref{e1.6a}. The second
relation is justified, for example, in \cite{Kw}.
Notice that   $c_0(d,\alpha)$ is of order  $(2 - \alpha)^{-1}$ as $\alpha \to 2$.
\end{proof}

From  Lemma  \ref{lem1.1} we deduce that the form  $a_0(\alpha)$ is closed. Then, in view of estimates
\eqref{e1.4}, the form $a(\alpha,\mu)$ is also closed.
The operator ${\mathbb A} = {\mathbb A}(\alpha,\mu)$ is defined as \emph{a non-negative self-adjoint
operator in $L_2(\R^d)$ that corresponds to the closed form }   \eqref{e1.3}.  Formally, one can write, see \cite{KaPiaZhi19},
\begin{equation*}
({\mathbb A} u) (\x) =  \intop_{\R^d} \mu(\x, \y) \frac{\left( u(\x) - u(\y) \right)}{|\x - \y|^{d+\alpha}}\,d\y.
\end{equation*}
Similarly, the self-adjoint operator  in $L_2(\R^d)$ that corresponds to the closed form  \eqref{e1.5} is denoted by  ${\mathbb A}_0 = {\mathbb A}_0(\alpha)$.
Due to representation \eqref{e1.6} the operator ${\mathbb A}_0(\alpha)$ coincides with the fractional Laplacian $( - \Delta)^{\gamma}$
up to a multiplicative constant:
\begin{equation*}
{\mathbb A}_0(\alpha)  = c_0(d,\alpha) ( - \Delta)^{\gamma},  \quad \Dom{\mathbb A}_0(\alpha) = H^{\alpha}(\R^d),\quad
\gamma=\frac\alpha2.
\end{equation*}
It follows from representation \eqref{e1.6} that $\lambda_{0}=0$ is the edge point of the spectrum of operator
${\mathbb A}_0(\alpha)$.
By \eqref{e1.4} the forms $a(\alpha,\mu)$ and $a_0(\alpha)$ are comparable. Therefore, $\lambda_{0}=0$
is also the lower edge of the spectrum of operator  $\A(\alpha,\mu)$.

\subsection{Operator family $\A(\bxi;\alpha,\mu)$}\label{sec1.2}
Denote by $\Omega:=[0,1)^{d}$ the periodicity cell of the lattice $\Z^d$, then  $\widetilde\Omega:=[-\pi,\pi)^{d}$ is the periodicity cell of the dual lattice $(2\pi\mathbb{Z})^{d}$.
For $s>0$ the space $\wt{H}^s(\Omega)$ is defined as a space of functions from $H^s(\Omega)$  whose
$\Z^d$-periodic extension belongs to $H^s_{\operatorname{loc}}(\R^d)$.
The standard unitary discrete Fourier transform ${\mathcal F}: L_{2}(\Omega)\to \ell_{2}(\Z^d)$ is defined by
\begin{gather*}
{\mathcal F} u(\n)= \wh{u}_\n = \intop_{\Omega}u(\x)e^{-2\pi i \langle \n, \x \rangle}d\x,\ \ \n\in\Z^d,\ \ u\in L_{2}(\Omega);
\\
u(\x) = \sum_{\n\in\Z^d} \wh{u}_{\n}e^{2\pi i \langle \n, \x\rangle},\ \ \x\in\Omega.
\end{gather*}
Then the relation $u \in \wt{H}^s(\Omega)$ is equivalent to the convergence of the series
$$
\sum_{\n \in \Z^d} (1 + |2\pi \n|^2)^s | \wh{u}_\n|^2,
$$
moreover, the latter expression admits two-sided estimates through $\| {u} \|^2_{{H}^s(\Omega)}$.

In the space $L_2(\Omega)$ we introduce a family of quadratic forms $a(\bxi)= a(\bxi;\alpha,\mu)$,  $\bxi \in \widetilde{\Omega}$, that read
 \begin{equation}
\label{e1.9}
a(\bxi;\alpha,\mu) [u,u] := \frac{1}{2} \intop_{\R^d} d\y \intop_{\Omega} d\x\, \mu(\x,\y)
\frac{| e^{i \langle \bxi,\x \rangle}u(\x)- e^{i \langle \bxi,\y \rangle} u(\y)|^2}{|\x - \y|^{d+\alpha}},\ \
u\in \wt{H}^\gamma(\Omega);
\end{equation}
here and in what follows we identify the functions $u\in \wt{H}^\gamma(\Omega)$ with their $\Z^d$-periodic extensions.
Due to \eqref{e1.1} the forms  $a(\bxi;\alpha,\mu)$ are densely defined and  non-negative,  moreover, they satisfy the following estimates:
 \begin{equation}
\label{e1.10}
\mu_- a_0(\bxi;\alpha) [u,u] \le a(\bxi;\alpha,\mu) [u,u] \le \mu_+ a_0(\bxi;\alpha) [u,u],\ \ u\in \wt{H}^\gamma(\Omega);
\end{equation}
here
\begin{equation}
\label{e1.11}
a_0(\bxi;\alpha) [u,u] := \frac{1}{2} \intop_{\R^d} d\y \intop_{\Omega} d\x\,
\frac{| e^{i \langle \bxi,\x \rangle}u(\x)- e^{i \langle \bxi,\y \rangle} u(\y)|^2}{|\x - \y|^{d+\alpha}},\ \
u\in \wt{H}^\gamma(\Omega).
\end{equation}

\begin{lemma}
\label{lem1.2}
Assume that $0< \alpha < 2$. Then the form  \eqref{e1.11} admits the representation
\begin{equation}
\label{e1.12}
a_0(\bxi;\alpha) [u,u] =  c_0(d,\alpha) \sum_{\n \in \Z^d} | 2\pi \n + \bxi|^{\alpha} | \wh{u}_\n|^2,\ \
u\in \wt{H}^\gamma(\Omega),
\end{equation}
where $\wh{u}_\n$, $\n \in \Z^d$, are the Fourier coefficients of function $u$, and  $c_0(d,\alpha)$
is the constant introduced in \eqref{e1.6a}.
\end{lemma}

\begin{proof}
Making the change of variables $\y \mapsto \z = \y-\x$ and using the unitarity of the discrete Fourier transform in $L_2(\Omega)$, we arrive at the relation
\begin{equation*}
\begin{aligned}
a_0(\bxi;\alpha) [u,u] &= \frac{1}{2} \intop_{\R^d} d\z \intop_{\Omega} d\x\, \frac{|u(\x)- e^{i \langle \bxi,\z\rangle} u(\x+\z)|^2}{|\z|^{d+\alpha}}
\\
&= \sum_{\n \in \Z^d}  | \wh{u}_\n|^2 \frac{1}{2} \intop_{\R^d} d\z\, \frac{ |1 -e^{i \langle 2\pi \n + \bxi,\z\rangle}|^2}{|\z|^{d+\alpha}}
=  \sum_{\n \in \Z^d}  V_\alpha(2\pi \n + \bxi) | \wh{u}_\n|^2
\end{aligned}
\end{equation*}
with  $V_\alpha(2\pi \n + \bxi)$ defined in  \eqref{e1.6b}. It remains to combine the latter relation with  \eqref{e1.6c}.
\end{proof}

 The last statement implies the closedness of the form $a_0(\bxi;\alpha)$ and, in view of estimates \eqref{e1.10},
 also the closedness of $a(\bxi;\alpha,\mu)$.
 
By definition ${\mathbb A}(\bxi) = {\mathbb A}(\bxi;\alpha,\mu)$ \emph{is the self-adjoint operator in $L_2(\R^d)$ that corresponds to the closed form}   \eqref{e1.9}. Formally, it takes the form
\begin{equation*}
({\mathbb A}(\bxi) u) (\x) =  \intop_{\R^d} \mu(\x, \y) \frac{\left( u(\x) - e^{- i \langle \bxi, \x-\y\rangle}u(\y) \right)}{|\x - \y|^{d+\alpha}}\,d\y.
\end{equation*}
Let  ${\mathbb A}_0(\bxi) = {\mathbb A}_0(\bxi;\alpha)$ be the self-adjoint operator generated by the closed form in
\eqref{e1.11}.  Thanks to \eqref{e1.12} this operator coincides with the fractional power of the operator $|\D +\bxi|$
up to a multiplicative constant:
\begin{equation}
\label{e1.14}
{\mathbb A}_0(\bxi;\alpha)  = c_0(d,\alpha) | \D + \bxi|^{\alpha},  \quad \Dom{\mathbb A}_0(\bxi;\alpha) =
\wt{H}^{\alpha}(\Omega).
\end{equation}
Since the domain of the form $a(\bxi;\alpha,\mu)$ is $\wt{H}^{\gamma}(\Omega)$, then, due to the compactness
of embedding of $\wt{H}^{\gamma}(\Omega)$ in $L_2(\Omega)$, the spectrum of operator $\A(\bxi;\alpha,\mu)$,
as well as that of operator $\A_0(\bxi;\alpha)$, is discrete for any $\bxi \in \wt{\Omega}$.

\subsection{Representation of the operator  $\A(\alpha,\mu)$ as a direct integral}
For $\n\in\Z^d$ denote  by  $S_{\n}$ the unitary shift operator in  $L_2(\R^d)$
defined by
\begin{equation*}
S_{\n}u(\x) =u(\x+\n),\ \ \x\in\R^d,\ \ u \in L_2(\R^d).
\end{equation*}
It is clear that under conditions  (\ref{e1.1}), (\ref{e1.2}) we have
$$
a(\alpha,\mu) [S_{\n} u, S_{\n}u] = a(\alpha,\mu) [ u, u],\quad u \in H^\gamma(\R^d), \quad \n \in \Z^d.
$$
This relation implies that the operator $\A(\alpha,\mu)$ commutes with the operators $S_{\n}$ for all $\n \in \Z^d$,
that is $\A(\alpha,\mu)$ is a  \emph{periodic operator}.

The Gelfand transform $\mathcal{G}$  (see  \cite{Sk} or \cite[Ch.~2]{BSu1}) is first defined for the Schwartz class functions as follows:
\begin{equation*}
\mathcal{G}u(\bxi,\x) = \wt{u}(\bxi,\x):=(2\pi)^{-d/2}\sum_{\n\in\Z^d} u(\x+\n) e^{-i \langle \bxi, \x+\n\rangle},\
\ \bxi\in\widetilde\Omega,\ \ \x\in\Omega,\ \ u\in {\mathcal S}(\R^d).
\end{equation*}
This transform extends by continuity to a unitary operator
 $$
 \mathcal{G}: L_{2}(\R^d) \to \intop_{\widetilde\Omega}\oplus L_{2}(\Omega)\, d\bxi=L_{2}(\widetilde\Omega\times\Omega).
 $$
Recall that the Gelfand transform maps the Sobolev space $H^s(\R^d)$,  $s>0$, onto
direct integral of the spaces $\wt{H}^s(\Omega)$:
 $$
 \mathcal{G}: H^s(\R^d) \to \intop_{\widetilde\Omega}\oplus \wt{H}^s(\Omega)\, d\bxi
 =L_{2}(\widetilde\Omega; \wt{H}^s(\Omega)).
 $$

Like any periodic operator,  $\A(\alpha,\mu)$  can be expanded into a direct integral by means of the Gelfand
transform.  %, i.e. $\A(\alpha,\mu)$ admits partial diagonalization.
This is the subject of the following statement.
\begin{lemma}
\label{lem1.3}
Let conditions  \eqref{e1.1} and \eqref{e1.2} be fulfilled, and assume that $0< \alpha <2$.
Then $u \in H^\gamma(\R^d)$ if and only if $\mathcal{G} u = \wt{u} \in L_2(\wt{\Omega}; \wt{H}^\gamma(\Omega))$,
where ${\mathcal G} : L_2(\R^d) \to L_2(\wt{\Omega} \times \Omega)$
is the unitary Gelfand transform.  Moreover,
\begin{equation}
\label{e1.14aa}
 a[u,u] = \intop_{\wt{\Omega}} a(\bxi) [ \wt{u}(\bxi,\cdot), \wt{u}(\bxi,\cdot)] \, d\bxi,
 \quad u \in H^\gamma(\R^d),
 \end{equation}
where  the form $a = a(\alpha,\mu)$ is defined in \eqref{e1.3} and the family of forms $a(\bxi) = a(\bxi;\alpha,\mu)$,
$\bxi \in \wt{\Omega}$, is given by \eqref{e1.9}.
\end{lemma}
\begin{proof}
For $u \in H^\gamma(\R^d)$ we rewrite the integral with respect to the variable $\x$ in \eqref{e1.3}
as the sum of integrals over the cells $\Omega + \n$, $\n \in \Z^d$. Then, making the change of
variables $\y \mapsto \y+\n$  in these integrals and considering the periodicity of  function  $\mu$
we conclude that
\begin{equation}
\label{e1.14a}
a[u,u] = \sum_{\n \in \Z^d} \frac{1}{2} \intop_{\R^d} d\y \intop_{\Omega} d\x\, \mu(\x,\y) \frac{|u(\x+\n)-u(\y+\n)|^2}{|\x - \y|^{d+\alpha}}.
\end{equation}
Applying the inverse Gelfand transform to the function  $\wt{u}(\bxi,\x)$ yields
$$
\begin{aligned}
u(\x + \n)  = (2\pi)^{-d/2} \intop_{\wt{\Omega}} e^{i\langle \bxi, \x + \n\rangle} \wt{u}(\bxi,\x) \, d\bxi.
%\\
 %u(\y + \n) = (2\pi)^{-d/2} \intop_{\wt{\Omega}} e^{i\langle \bxi, \y + \n\rangle} \wt{u}(\bxi,\y) \, d\bxi.
\end{aligned}
$$
Therefore,
$$
u(\x + \n) - u(\y + \n) = (2\pi)^{-d/2} \intop_{\wt{\Omega}} e^{i\langle \bxi, \n\rangle} \left( e^{i\langle \bxi, \x\rangle} \wt{u}(\bxi,\x)
- e^{i\langle \bxi, \y\rangle} \wt{u}(\bxi,\y) \right) \, d\bxi,\quad \n \in \Z^d.
$$
This implies that for each fixed $\x$ and $\y$ the sequence $u(\x + \n) - u(\y + \n)$, $\n \in \Z^d$,
consists of the Fourier coefficients of $(2\pi \Z)^d$-periodic function $e^{i\langle \bxi, \x\rangle} \wt{u}(\bxi,\x)
- e^{i\langle \bxi, \y\rangle} \wt{u}(\bxi,\y)$. By the Parsevale identity we have
$$
\sum_{\n \in \Z^d} |u(\x+\n)-u(\y+\n)|^2 = \intop_{\wt{\Omega}} \left| e^{i\langle \bxi, \x\rangle} \wt{u}(\bxi,\x)
- e^{i\langle \bxi, \y\rangle} \wt{u}(\bxi,\y)\right|^2 \,d\bxi.
$$
Combining this identity with  \eqref{e1.14a} we arrive at \eqref{e1.14aa}.
\end{proof}

Since the operators $\A$ and  $\A(\bxi)$ are defined in terms of the quadratic forms introduced in \eqref{e1.3}
and  \eqref{e1.9},  it follows from Lemma \ref{lem1.3} that
\begin{equation}
\label{e1.15}
\A(\alpha,\mu) = {\mathcal G}^* \Bigl( \int_{\widetilde\Omega} \oplus  \A(\bxi;\alpha,\mu) \,d\bxi\Bigr) {\mathcal G}. \end{equation}

\subsection{Estimates for the quadratic form  of the operator $\A(\bxi;\alpha,\mu)$}\label{sec1.4}

According to Lemma \ref{lem1.2} the operators $\A_{0}(\bxi; \alpha)$, $\bxi\in\widetilde\Omega$, are diagonalized by
the unitary discrete Fourier transform ${\mathcal F}$:
\begin{equation}
\label{e1.16}
\A_{0}(\bxi; \alpha)= c_0(d,\alpha) {\mathcal F}^{*} \bigl[   |2\pi \n+\bxi|^\alpha \bigr] {\mathcal F},\ \ \bxi\in\wt{\Omega};
\end{equation}
here $[  |2\pi \n+\bxi|^\alpha ]$  stands for the operator of multiplication by the function
$ |2\pi \n+\bxi|^\alpha$, $\n \in \Z^d$, that acts in the space $\ell_2(\Z^d)$.

Letting $\bxi=0$ in \eqref{e1.16} one can easily deduce that
$\operatorname{Ker}\A_{0}(\mathbf{0};\alpha)=\mathcal{L}\{\mathbf{1}_{\Omega}\}$. In view of \eqref{e1.10}
this yields $\operatorname{Ker}\A(\mathbf{0}; \alpha,\mu)=\mathcal{L}\{\mathbf{1}_{\Omega}\}$,
and we arrive at the following statement:
\begin{lemma}
\label{lem1.4}
Let conditions \eqref{e1.1} and \eqref{e1.2} be fulfilled, and assume that  $0< \alpha <2$.
Then $\lambda_{0}=0$ is a simple eigenvalue of the operator $\A(\mathbf{0};\alpha,\mu)$,
and $\operatorname{Ker}\A(\mathbf{0};\alpha,\mu)=\mathcal{L}\{\mathbf{1}_{\Omega}\}$.
\end{lemma}

It is straightforward to show that
\begin{align}
\label{1.18a}
& |2\pi \n+\bxi|^\alpha \ge |\bxi|^\alpha, \quad \bxi\in\widetilde\Omega,\ \ \n\in\Z^d,
\\
\label{1.18b}
& |2\pi \n+\bxi|^\alpha \ge \pi^\alpha, \quad \bxi\in\widetilde\Omega,\ \ \n\in\Z^d \setminus{\mathbf{0}},
\\
\label{1.18c}
& \min_{\n\in\Z^d \setminus{\mathbf{0}}} |2\pi \n|^\alpha = (2\pi)^\alpha.
\end{align}
Combining \eqref{1.18a}, \eqref{1.18b} and the statement of Lemma \ref{lem1.2} we obtain the following estimates for the  quadratic form $a_0(\bxi;\alpha)$:
\begin{align}
\label{e1.17}
 a_0(\bxi;\alpha)[u,u]& \ge c_0(d,\alpha) |\bxi|^{\alpha} \| u \|_{L_2(\Omega)}^2,\ \ u \in \wt{H}^\gamma(\Omega), \quad \bxi\in\widetilde\Omega,
 \\
 \label{e1.18}
 a_0(\bxi;\alpha)[u,u] &\ge c_0(d,\alpha) \pi^{\alpha} \| u \|_{L_2(\Omega)}^2,\ \ u \in \wt{H}^\gamma(\Omega), \ \ \intop_\Omega u(\x)\,d\x =0, \quad \bxi\in\widetilde\Omega.
 \end{align}
As a consequence of \eqref{e1.10}, \eqref{e1.17} and  \eqref{e1.18} one has
\begin{proposition}
\label{prop1.5}
Let conditions \eqref{e1.1} and \eqref{e1.2} be fulfilled, and assume that $0< \alpha <2$.
Then the form in \eqref{e1.9} satisfies the estimates
\begin{align*}
 a(\bxi;\alpha,\mu)[u,u] &\ge \mu_- c_0(d,\alpha) |\bxi|^{\alpha} \| u \|_{L_2(\Omega)}^2,\ \ u \in \wt{H}^\gamma(\Omega), \quad \bxi\in\widetilde\Omega,
 \\
 a(\bxi;\alpha,\mu)[u,u] &\ge \mu_- c_0(d,\alpha) \pi^{\alpha} \| u \|_{L_2(\Omega)}^2,\ \ u \in \wt{H}^\gamma(\Omega), \quad \intop_\Omega u(\x)\,d\x =0, \quad  \bxi\in\widetilde\Omega.
 \end{align*}
\end{proposition}

\section{Estimates for the difference of  the quadratic forms $a(\bxi)$ and $a(\mathbf{0})$}

In this section we establish a number of estimates for the difference between the quadratic forms
$a(\bxi)$ and $a(\mathbf{0})$. These estimates are used later on.
It turns out that the case $0 < \alpha <1$ differs essentially from the case $1 \le \alpha < 2$.
We consider these two cases separately.

\subsection{The case $0< \alpha <1$}
In what follows we use for the sesquilinear form corresponding to quadratic form   \eqref{e1.9}
the same notation $a(\bxi)[u,v]$, $u,v \in \wt{H}^\gamma(\Omega)$, as for the quadratic form.
\begin{lemma}
\label{lem2.2}
Let conditions \eqref{e1.1} and \eqref{e1.2} be fulfilled, and assume that $0< \alpha <1$.
Then, for $\bxi \in \wt{\Omega}$, the formula for
the difference of sesquilinear forms $a(\bxi) - a(\mathbf{0})$ reads %satisfies the relation
\begin{equation*}
\label{e2.9}
a(\bxi)[u,v] - a({\mathbf 0})[u,v] =
  \sum_{\n\in \Z^d} \intop_{\Omega} d\y \intop_\Omega d\x \,\mu(\x,\y)
 \frac{ u(\x)\overline{v(\y)} (1 - e^{i\langle \bxi, \x +\n-\y\rangle}) }
 {|\x +\n-\y |^{d+\alpha}}, \quad u,v \in \wt{H}^\gamma(\Omega).
\end{equation*}
\end{lemma}

\begin{proof}
We identify functions $u,v \in \wt{H}^\gamma(\Omega)$ with their periodic extensions in $\R^d$.
Taking into account the periodicity of functions $u$, $v$ and $\mu$ and the symmetry condition  $\mu(\x,\y) = \mu(\y,\x)$
we have
$$
\begin{aligned}
a(\bxi)[u,v] \!-\! a({\mathbf 0})[u,v] &= \frac{1}{2} \intop_{\R^d} d\y \intop_\Omega d\x \,\mu(\x,\y)
 \frac{ (u(\x) - e^{i\langle \bxi, \y-\x\rangle}u(\y)) (\overline{v(\x)} - e^{-i\langle \bxi, \y-\x\rangle}\overline{v(\y)})}{|\x-\y|^{d+\alpha}}
\\
&- \frac{1}{2} \intop_{\R^d} d\y \intop_\Omega d\x \,\mu(\x,\y)
 \frac{ (u(\x) - u(\y)) (\overline{v(\x)} - \overline{v(\y)})}{|\x-\y|^{d+\alpha}}
 \\
 &= \frac{1}{2} \intop_{\R^d} d\y \intop_\Omega d\x \,\mu(\x,\y)
 \frac{ (u(\x)\overline{v(\y)} (1 - e^{-i\langle \bxi, \y-\x\rangle}) + u(\y)\overline{v(\x)} (1 - e^{i\langle \bxi, \y-\x\rangle})}
 {|\x-\y|^{d+\alpha}}
 \\
 &= \sum_{\n\in \Z^d}\frac{1}{2} \intop_{\Omega} d\y \intop_\Omega d\x \,\mu(\x,\y)
 \frac{ u(\x)\overline{v(\y)} (1 - e^{-i\langle \bxi, \y -\n-\x\rangle}) }
 {|\x -\y + \n|^{d+\alpha}}
 \\
 &+ \sum_{\n\in \Z^d}\frac{1}{2} \intop_{\Omega} d\y \intop_\Omega d\x \,\mu(\x,\y)
 \frac{  u(\y)\overline{v(\x)} (1 - e^{i\langle \bxi, \y+\n-\x\rangle})}
 {|\x - \n -\y|^{d+\alpha}}
 \\
 &=
  \sum_{\n\in \Z^d} \intop_{\Omega} d\y \intop_\Omega d\x \,\mu(\x,\y)
 \frac{ u(\x)\overline{v(\y)} (1 - e^{-i\langle \bxi, \y -\n-\x\rangle}) }
 {|\x -\y + \n|^{d+\alpha}}.
\end{aligned}
$$
\end{proof}
As a consequence of Lemma \ref{lem2.2} we obtain
\begin{lemma}
\label{lem2.3}
For $0< \alpha < 1$ the operator $\Delta \A(\bxi) = \Delta \A(\bxi;\alpha,\mu):= \A(\bxi;\alpha,\mu) - \A(\mathbf{0};\alpha,\mu)$ is well-defined and bounded in $L_2(\Omega)$. It is an integral operator with
the kernel
\begin{equation}
\label{e2.11}
K(\x,\y;\bxi) := \mu(\x,\y) \sum_{\n\in \Z^d}
 \frac{ (1 - e^{i\langle \bxi, \x +\n-\y\rangle})}{|\x +\n-\y |^{d+\alpha}},\quad \x,\y \in \Omega,\ \bxi \in \wt{\Omega}.
\end{equation}
The estimate
\begin{equation}
\label{e2.10}
\| \Delta \A(\bxi)\|_{L_2(\Omega) \to L_2(\Omega)} \le  \mu_+ c_1(d,\alpha) |\bxi|^\alpha
\end{equation}
holds, where the constant  $c_1(d,\alpha)$ is defined in  \eqref{e2.12} below.
\end{lemma}

\begin{proof}
We have
\begin{equation}\label{andr1}
\begin{aligned}
\intop_\Omega |K(\x,\y;\bxi)| \,d\y \le \mu_+ \sum_{\n \in \Z^d}
\intop_\Omega  \frac{ |1 - e^{i\langle \bxi, \x +\n-\y\rangle}|}{|\x +\n-\y |^{d+\alpha}} \,d\y
= \mu_+
\intop_{\R^d}  \frac{ |1 - e^{i\langle \bxi, \x -\y\rangle}|}{|\x -\y |^{d+\alpha}} \,d\y
\\
= \mu_+
\intop_{\R^d}  \frac{ |1 - e^{i\langle \bxi, \z\rangle}|}{|\z |^{d+\alpha}} \,d\z
= \mu_+ \intop_{\R^d}  \frac{ 2 \left|\sin \frac{\langle \bxi, \z\rangle}{2}\right|}{|\z |^{d+\alpha}} \,d\z
=  \mu_+ c_1(d,\alpha) |\bxi|^\alpha
\end{aligned}
\end{equation}
with
%где
\begin{equation}
\label{e2.12}
c_1(d,\alpha) := \intop_{\R^d}  \frac{ 2 \left|\sin \frac{z_1}{2}\right|}{|\z |^{d+\alpha}} \,d\z < \infty,\quad 0< \alpha <1.
\end{equation}
Since $0< \alpha <1$, the integral on the right-hand side of \eqref{e2.12} converges. Notice however that
 $c_1(d,\alpha) = O((1-\alpha)^{-1})$ as $\alpha \to 1$.
From \eqref{andr1} it follows that
$$
\sup_{\x \in \Omega} \intop_\Omega |K(\x,\y;\bxi)| \,d\y \le  \mu_+ c_1(d,\alpha) |\bxi|^\alpha.
$$
The quantity $\sup_{\y \in \Omega} \intop_\Omega |K(\x,\y;\bxi)| \,d\x$ admits a similar estimate.

Denote by $\mathbb{K}(\bxi)$ the integral operator in $L_2(\Omega)$ with the kernel $K(\x,\y;\bxi)$ defined in \eqref{e2.11}.
By the Shur test $\mathbb{K}(\bxi)$ is a bounded operator in $L_2(\Omega)$, and its norm can be estimated as follows:
\begin{equation}
\label{e2.13}
\| \mathbb{K}(\bxi) \|_{L_2(\Omega) \to L_2(\Omega)} \le  \mu_+ c_1(d,\alpha) |\bxi|^\alpha,\quad
\bxi \in \wt{\Omega},\quad 0< \alpha <1.
\end{equation}
Due to Lemma \ref{lem2.2}
$$
a(\bxi)[u,v] - a({\mathbf 0})[u,v] =
   \intop_{\Omega} d\y \intop_\Omega d\x \,K(\x,\y;\bxi) u(\x)\overline{v(\y)}
   = ({\mathbb K}(\bxi)u,v)_{L_2(\Omega)}, \quad u,v \in \wt{H}^\gamma(\Omega).
$$
Since the operator $\mathbb{K}(\bxi)$ is bounded,
the operators  $\A(\bxi)$ and $\A(\mathbf{0})$ have the same domain, and
$$
\A(\bxi) u - \A(\mathbf{0}) u =   {\mathbb K}(\bxi)u, \quad u \in \Dom \A(\bxi) = \Dom \A(\mathbf{0}).
$$
Considering the fact that $\Dom \A(\bxi) = \Dom \A(\mathbf{0})$ is dense in $L_2(\Omega)$ one can extend
the operator $\Delta \A(\bxi) :=\A(\bxi)  - \A(\mathbf{0})$ to the whole $L_2(\Omega)$. The extended operator
is bounded and coincides with  $\mathbb{K}(\bxi)$. Finally, estimate \eqref{e2.10} follows from \eqref{e2.13}.
\end{proof}

\subsection{The case $1 \le \alpha <2$}

\begin{lemma}
\label{lem2.4}
Assume that  $1 \le \alpha < 2$.  Then
\begin{equation}
\label{*.0}
\left|a(\bxi) [u,u] - a(\mathbf{0}) [u,u] \right| \le \check{c}(d,\alpha) \Theta(\bxi) \left( a(\mathbf{0}) [u,u] + \mu_+ \|u\|^2_{L_2(\Omega)}\right), \quad u\in \wt{H}^\gamma(\Omega), \quad \bxi \in \wt{\Omega},
\end{equation}
where
\begin{equation}
\label{*.00}
\Theta(\bxi) := \begin{cases} |\bxi|, &\hbox{if } 1< \alpha <2, \\ |\bxi| \left(1+ |\operatorname{ln} |\bxi| | \right), 
&\hbox{if } \alpha=1. \end{cases}
\end{equation}
\end{lemma}

\begin{proof}
According to \eqref{e1.9}, for all  $u\in \wt{H}^\gamma(\Omega)$ it holds
 \begin{equation*}
a(\bxi) [u,u] - a(\mathbf{0}) [u,u] := \frac{1}{2} \intop_{\R^d} d\y \intop_{\Omega} d\x\, \mu(\x,\y)
\frac{ \left( | u(\x)- e^{i \langle \bxi,\y -\x\rangle} u(\y)|^2 - | u(\x)- u(\y)|^2 \right)}{|\x - \y|^{d+\alpha}}.
\end{equation*}
Due to an evident relation
$$
 | u(\x)- e^{i \langle \bxi,\y -\x\rangle} u(\y)|^2 - | u(\x)- u(\y)|^2
 = | u(\x)|^2 |1 - e^{i \langle \bxi,\y -\x\rangle}|^2 +
  2 \operatorname{Re}( u(\x)- u(\y)) \overline{u(\x)} ( e^{i \langle \bxi,\y -\x\rangle} -1)
 $$
 we have
 \begin{equation}
\label{*.2}
a(\bxi) [u,u] - a(\mathbf{0}) [u,u] = b_1(\bxi) [u,u] + b_2(\bxi) [u,u],
\end{equation}
where
\begin{align}
\label{*.3}
b_1(\bxi) [u,u] := &  \frac{1}{2} \intop_{\R^d} d\y \intop_{\Omega} d\x\, \mu(\x,\y)
\frac{  | u(\x)|^2 | 1 - e^{i \langle \bxi,\y -\x\rangle} |^2}{|\x - \y|^{d+\alpha}},
\\
\label{*.4}
b_2(\bxi) [u,u] := &  \operatorname{Re} \intop_{\R^d} d\y \intop_{\Omega} d\x\, \mu(\x,\y)
\frac{ ( u(\x)- u(\y)) \overline{u(\x)} ( e^{i \langle \bxi,\y -\x\rangle} -1)}{ |\x - \y|^{d+\alpha} }.
\end{align}
First, we estimate the quadratic form in  \eqref{*.3}:
\begin{equation}
\label{*.5}
b_1(\bxi) [u,u] \le \frac{\mu_+}{2}  \intop_{\Omega} d\x\,| u(\x)|^2  \intop_{\R^d} d\z
\frac{  | 1 - e^{i \langle \bxi,\z \rangle} |^2}{  |\z|^{d+\alpha}} = \mu_+ V_\alpha(\bxi) \|u\|^2_{L_2(\Omega)}
= \mu_+ c_0(d,\alpha) |\bxi |^\alpha  \|u\|^2_{L_2(\Omega)};
\end{equation}
here the notation from \eqref{e1.6b} and relation \eqref{e1.6c} have been used.

We turn to estimating the form in  \eqref{*.4}:
\begin{equation}
\label{*.6}
 | b_2(\bxi) [u,u] |  \le \
  \intop_{\R^d} d\y \intop_{\Omega} d\x\, \mu(\x,\y)
\frac{ |u(\x)- u(\y)| | u(\x)| | e^{i \langle \bxi,\y -\x\rangle} -1|}{ |\x - \y|^{d+\alpha} }=
 J_1(\bxi)[u] + J_2(\bxi)[u]
\end{equation}
with
\begin{align}
\label{*.7}
J_1(\bxi)[u] :=&  \intop_{[-2,2]^d} d\y \intop_{\Omega} d\x\,  \mu(\x,\y)
\frac{ |u(\x)- u(\y)| | u(\x)| | e^{i \langle \bxi,\y -\x\rangle} -1|}{ |\x - \y|^{d+\alpha} },
\\
\label{*.8}
J_2(\bxi)[u] := & \intop_{\R^d \setminus [-2,2]^d} d\y \intop_{\Omega} d\x\,  \mu(\x,\y)
\frac{ |u(\x)- u(\y)| | u(\x)| | e^{i \langle \bxi,\y -\x\rangle} -1|}{ |\x - \y|^{d+\alpha} }.
\end{align}
By the Cauchy inequality, the form in \eqref{*.8} satisfies the estimate
%В силу неравенства Коши для  формы \eqref{*.8} справедлива оценка
\begin{equation}
\label{*.9}
J_2(\bxi)[u] \le \mu_+ \left(J_2^{(1)}(\bxi)[u]\right)^{1/2} \left(J_2^{(2)}(\bxi)[u]\right)^{1/2},
\end{equation}
where
%где
\begin{align}
\label{*.10}
J_2^{(1)}(\bxi)[u] := &
 \intop_{\R^d \setminus [-2,2]^d} d\y \intop_{\Omega} d\x\,
\frac{ |u(\x)- u(\y)|^2  | e^{i \langle \bxi,\y -\x\rangle} -1|}{ |\x - \y|^{d+\alpha} },
\\
\label{*.11}
J_2^{(2)}(\bxi)[u] := & \intop_{\R^d \setminus [-2,2]^d} d\y \intop_{\Omega} d\x\,
\frac{ |u(\x)|^2  | e^{i \langle \bxi,\y -\x\rangle} -1|}{ |\x - \y|^{d+\alpha} }.
\end{align}
The form on the right-hand side of   \eqref{*.11} can be estimated as follows:
\begin{equation}
\label{*.12}
J_2^{(2)}(\bxi)[u] =  \intop_{\Omega} d\x\, |u(\x)|^2
\intop_{\R^d \setminus [-2,2]^d} d\y \frac{ | e^{i \langle \bxi,\y -\x\rangle} -1|}{ |\x - \y|^{d+\alpha} }
\le   \intop_{\Omega} d\x\, |u(\x)|^2
\intop_{|\z| \ge 1} d\z \frac{ | e^{i \langle \bxi,\z \rangle} -1|}{ |\z |^{d+\alpha} }.
\end{equation}
Making the change of variables $\w= |\bxi| \z$ and letting $\bxi = |\bxi| \widehat{\bxi}$ we obtain
\begin{equation*}
\intop_{|\z| \ge 1} d\z \frac{ | e^{i \langle \bxi,\z \rangle} -1|}{ |\z |^{d+\alpha} } = |\bxi|^\alpha
\intop_{|\w| \ge |\bxi|} d\w \frac{ | e^{i \langle \widehat{\bxi},\w \rangle} -1|}{ |\w |^{d+\alpha} }
= |\bxi|^\alpha \intop_{|\w| \ge |\bxi|} d\w \frac{ | e^{i w_1 } -1|}{ |\w |^{d+\alpha} }.
\end{equation*}
If $|\bxi| \ge 1$, then
$$
 \intop_{|\w| \ge |\bxi|} d\w \frac{ | e^{i w_1 } -1|}{ |\w |^{d+\alpha} }
 \le \intop_{|\w| \ge 1} d\w \frac{ | e^{i w_1 } -1|}{ |\w |^{d+\alpha} } \le \frac{2 \omega_d}{\alpha},
$$
where $\omega_d$ is the area of the unit sphere in  $\R^d$. If $|\bxi| \le 1$, then
$$
 \intop_{|\w| \ge |\bxi|} d\w \frac{ | e^{i w_1 } -1|}{ |\w |^{d+\alpha} }
 \le \frac{2 \omega_d}{\alpha}  + \intop_{|\bxi| \le |\w| \le 1} d\w \frac{ | e^{i w_1 } -1|}{ |\w |^{d+\alpha} } \le
 \frac{2 \omega_d}{\alpha} +2 \!\! \intop_{|\bxi| \le |\w| \le 1} d\w \frac{1}{ |\w |^{d+\alpha -1} }.
$$
The straightforward computation yields
$$
 \intop_{|\bxi| \le |\w| \le 1} d\w \frac{1}{ |\w |^{d+\alpha -1} } \le
 \begin{cases}  \frac{\omega_d}{\alpha - 1} |\bxi|^{1-\alpha}, &\hbox{if } 1< \alpha <2, \\ \omega_d |\operatorname{ln} |\bxi| |, &\hbox{if } \alpha=1, \end{cases}
$$
and finally we have
\begin{equation}
\label{*.13}
\intop_{|\z| \ge 1} d\z \frac{ | e^{i \langle \bxi,\z \rangle} -1|}{ |\z |^{d+\alpha} }
\le c_2(d,\alpha) \Theta(\bxi)
\end{equation}
with $\Theta(\bxi)$ defined in \eqref{*.00}.
Notice that $c_2(d,\alpha) = O((\alpha -1)^{-1})$ as $\alpha \to 1$.  For $\alpha=1$ we have
$c_2(d,1)=2 \omega_d$, however, in this case the function $\Theta$ is defined in a different way.
From  \eqref{*.12} and \eqref{*.13} it follows that
\begin{equation}
\label{*.14}
J_2^{(2)}(\bxi)[u] \le c_2(d,\alpha) \Theta(\bxi) \|u\|^2_{L_2(\Omega)}.
\end{equation}
We proceed to estimating the form $J_2^{(1)}$ defined in \eqref{*.10}:
\begin{equation}
\label{*.15}
J_2^{(1)}(\bxi)[u] \le
 \intop_{\R^d \setminus [-2,2]^d} d\y \intop_{\Omega} d\x\,
\frac{(2 |u(\x)|^2 +2| u(\y)|^2)  | e^{i \langle \bxi,\y -\x\rangle} -1|}{ |\x - \y|^{d+\alpha} } =
2 J_2^{(2)}(\bxi)[u] + 2 \widetilde{J}_2^{(1)}(\bxi)[u],
\end{equation}
where
\begin{equation*}
\begin{aligned}
 \widetilde{J}_2^{(1)}(\bxi)[u] &:=
  \intop_{\R^d \setminus [-2,2]^d} d\y \intop_{\Omega} d\x\,
\frac{| u(\y)|^2 | e^{i \langle \bxi,\y -\x\rangle} -1|}{ |\x - \y|^{d+\alpha} }
\\
&= \sum_{\n \in \Z^d: (\Omega + \n) \subset \R^d \setminus (-2,2)^d } \intop_{\Omega} d\y | u(\y)|^2 \intop_{\Omega} d\x\,
\frac{| e^{i \langle \bxi,\y +\n -\x\rangle} -1|}{ |\x - \y - \n|^{d+\alpha} }.
\end{aligned}
\end{equation*}
Taking into account \eqref{*.13} we obtain
$$
\widetilde{J}_2^{(1)}(\bxi)[u] \le  \intop_{\Omega} d\y | u(\y)|^2 \intop_{|\z| \ge 1} d\z\,
\frac{| e^{i \langle \bxi,\z\rangle} -1|}{ |\z|^{d+\alpha} } \le  c_2(d,\alpha) \Theta(\bxi) \|u\|^2_{L_2(\Omega)}.
$$
Combining the latter inequality with   \eqref{*.14} and  \eqref{*.15} yields
\begin{equation}
\label{*.17}
J_2^{(1)}(\bxi)[u] \le 4 c_2(d,\alpha) \Theta(\bxi) \|u\|^2_{L_2(\Omega)}.
\end{equation}
Considering estimates \eqref{*.14} and \eqref{*.17} we derive from \eqref{*.9} the following inequality
\begin{equation}
\label{*.18}
J_2(\bxi)[u] \le 2 \mu_+ c_2(d,\alpha) \Theta(\bxi) \|u\|^2_{L_2(\Omega)}.
\end{equation}
It remains to estimate the form $J_1(\bxi)$ defined in  \eqref{*.7}.  To this end, using the simple estimate
$| e^{i \langle \bxi,\y -\x\rangle} -1| \le |\bxi||\x-\y|$  and the Cauchy inequality, we obtain
\begin{equation}
\label{*.19}
\begin{aligned}
J_1(\bxi)[u] &\le  |\bxi| \intop_{[-2,2]^d} d\y \intop_{\Omega} d\x\,  \mu(\x,\y)
\frac{ |u(\x)- u(\y)| | u(\x)| }{ |\x - \y|^{d+\alpha -1} }
\\
& \le
|\bxi| \Bigl( \intop_{[-2,2]^d} d\y \intop_{\Omega} d\x\,  \mu(\x,\y)
\frac{ |u(\x)- u(\y)|^2  }{ |\x - \y|^{d+\alpha} } \Bigr)^{1/2} \Bigl(\mu_+ \intop_{[-2,2]^d} d\y \intop_{\Omega} d\x\,
\frac{ |u(\x)|^2  }{ |\x - \y|^{d+\alpha-2} } \Bigr)^{1/2}
\\
& \le |\bxi|  \left( \mu_+ c_3 (d,\alpha)\right)^{1/2} \left(a(\mathbf{0})[u,u] \right)^{1/2}
\|u\|_{L_2(\Omega)};
\end{aligned}
\end{equation}
here we have also used the elementary inequality
$$
 \intop_{[-2,2]^d} d\y \intop_{\Omega} d\x\,
\frac{ |u(\x)|^2  }{ |\x - \y|^{d+\alpha-2} }  \le \intop_{\Omega} d\x\, |u(\x)|^2 \intop_{|\z|\le 4\sqrt{d}} \frac{d\z}{|\z|^{d+\alpha -2}} =
c_3(d,\alpha) \|u\|_{L_2(\Omega)}^2
$$
with $c_3(d,\alpha) =\frac{ \omega_d (4 \sqrt{d})^{2 - \alpha}}{2-\alpha}$.
Inserting estimates   \eqref{*.18} and  \eqref{*.19} in \eqref{*.6} we obtain
\begin{equation}
\label{*.20}
 | b_2(\bxi) [u,u] |  \le  |\bxi|  \left( \mu_+ c_3 (d,\alpha)\right)^{1/2} \left(a(\mathbf{0})[u,u] \right)^{1/2}
\|u\|_{L_2(\Omega)}+2 \mu_+ c_2(d,\alpha) \Theta(\bxi) \|u\|^2_{L_2(\Omega)},
\end{equation}
and the desired estimate  \eqref{*.0} follows from \eqref{*.2}, \eqref{*.5} and \eqref{*.20}.
\end{proof}

\section{Threshold characteristics of L\'evy-type operators \\ near the bottom of the spectrum.}

\subsection{The edge of spectrum of the operator $\A(\bxi;\alpha,\mu)$}
Denote by $\lambda_j(\bxi)$, $j \in \N$, the eigenvalues of the operator $\A(\bxi)$ that are enumerated in non-decreasing
order taking into account the  multiplicities. The eigenvalues of the operator  $\A_0(\bxi)$ are denoted by
$\lambda_j^0(\bxi)$, $j \in \N$.
In view of \eqref{e1.10} the variational principle implies the inequalities
\begin{equation}
\label{e1.20}
  \mu_- \lambda_j^0(\bxi) \le \lambda_j(\bxi) \le \mu_+ \lambda_j^0(\bxi),
 \quad  j \in \N, \quad  \bxi\in\widetilde\Omega.
 \end{equation}
Thanks to the diagonalization, the eigenvalues of operator $\A_0(\bxi)$ admit the following explicit formula:
%$ \lambda_1^0(\bxi) =
$c_0(d,\alpha) |2\pi \n + \bxi|^\alpha$, $\n \in \Z^d$, the corresponding eigenfunctions are
$e^{2\pi i \langle \n, \x \rangle}$.   The first eigenvalue is given by
\begin{equation}
\label{e1.21}
  \lambda_1^0(\bxi) = c_0(d,\alpha) |\bxi|^\alpha, \quad \bxi \in \wt{\Omega},
 \end{equation}
 and  $\mathbf{1}_{\Omega}$ is the first eigenfunction.
Since
   $$
   |\bxi| < |2\pi \n + \bxi|, \quad \bxi \in \operatorname{Int} \wt{\Omega} = (-\pi,\pi)^d, \ \ \n \in \Z^d \setminus \mathbf{0},
   $$
then, for $\bxi \in (-\pi,\pi)^d$,  the first eigenvalue of operator  $\A_0(\bxi)$  is simple, and
$\mathcal{L}\{\mathbf{1}_{\Omega}\}$ is the corresponding eigenspace.
Due to  \eqref{1.18b} and \eqref{1.18c} the following relations hold:
\begin{align}
\label{e1.22}
  \lambda_2^0(\bxi) &=  c_0(d,\alpha) \min_{\n \in \Z^d \setminus \mathbf{0}}
   |2\pi \n + \bxi|^\alpha \ge c_0(d,\alpha) \pi^\alpha, \quad \bxi \in \wt{\Omega},
   \\
   \label{e1.22a}
   \lambda_2^0(\mathbf{0}) &=  c_0(d,\alpha) (2\pi)^\alpha,
  \end{align}
and,  as a consequence of \eqref{e1.20}--\eqref{e1.22a}, we obtain
\begin{align}
\label{e1.23}
 \mu_- c_0(d,\alpha) |\bxi|^\alpha \le \,&\lambda_1(\bxi) \le \mu_+ c_0(d,\alpha) |\bxi|^\alpha,
 \quad   \bxi\in\widetilde\Omega,
\\
\label{e1.24}
  &\lambda_2(\bxi) \ge \mu_- c_0(d,\alpha) \pi^\alpha =: d_0,
 \quad   \bxi\in\widetilde\Omega,
 \\
\label{e1.24a}
  &\lambda_2(\mathbf{0}) \ge \mu_- c_0(d,\alpha) (2\pi)^\alpha = 2^\alpha d_0.
\end{align}

According to Lemma \ref{lem1.4},  under conditions \eqref{e1.1} and \eqref{e1.2} the lower edge of spectrum
of the operator  $\A(\mathbf{0};\alpha,\mu)$ consists of a simple isolated eigenvalue  $\lambda_1(\mathbf{0}) =0$,
and the corresponding eigenspace  is $\mathcal{L}\{\mathbf{1}_{\Omega}\}$.
Due to \eqref{e1.24a} the distance from the point $\lambda_1(\mathbf{0}) =0$ to the remaining part of the spectrum
of operator  $\A(\mathbf{0};\alpha,\mu)$ is at least $2^\alpha d_0$.

Letting
%Положим
\begin{equation}
\label{delta0}
\delta_{0}(\alpha,\mu):=  \pi \left( \frac{\mu_{-}}{3 \mu_{+}}\right)^{1/\alpha},
\end{equation}
one can easily check that $\delta_{0}(\alpha,\mu) < \pi$ and therefore the ball $|\bxi| \le \delta_{0}(\alpha,\mu)$
is a subset of $\widetilde\Omega$.
As follows from estimates  \eqref{e1.23} and \eqref{e1.24}, for all  $|\bxi|\le \delta_0(\alpha, \mu)$
the first eigenvalue of the operator $\A(\bxi;\alpha,\mu)$ is located on the interval $[0,d_0/3]$,
while the remaining part of the spectrum belongs to the interval $[d_0,\infty)$.  Thus,
\begin{equation*}
\operatorname{rank}E_{\A(\bxi;\alpha,\mu)}[0,d_{0}/3]=1,\ \
\sigma(\A(\bxi;\alpha,\mu))\cap(d_{0}/3;d_{0})=\varnothing, \quad  |\bxi|\le \delta_0(\alpha, \mu),
\end{equation*}
and we arrive at the following statement:
\begin{proposition}
\label{prop2.1}
Let conditions \eqref{e1.1} and \eqref{e1.2} be satisfied, and assume that $0< \alpha < 2$. Denote $d_0 := \mu_- c_0(d,\alpha) \pi^\alpha$. Then for all $\bxi$ such that $|\bxi|\le\delta_{0}(\alpha,\mu)$ the spectrum of operator
$\A(\bxi) = \A(\bxi;\alpha,\mu)$ on the interval  $[0,d_{0}/3]$ consists of a simple eigenvalue, while the interval
 $(d_{0}/3,d_{0})$ does not contain points of the spectrum of the operator $\A(\bxi;\alpha,\mu)$.
\end{proposition}

\subsection{Threshold approximations}
Denote by $F(\bxi)$ the spectral projection of the operator $\A(\bxi;\alpha,\mu)$ that corresponds to the interval $[0,d_{0}/3]$. Letting $\mathfrak{N}$ be the kernel $\operatorname{Ker}\A(\mathbf{0};\alpha,\mu)=\mathcal{L}\{\mathbf{1}_{\Omega}\}$,
we denote by  $P$ the orthogonal projection onto $\mathfrak{N}$; then $P = (\cdot, \1_\Omega)\1_\Omega$.
Let $\Gamma$ be a contour in the complex plane that passes through the middle point of the interval $(d_{0}/3,d_{0})$
and encloses the segment $[0,d_{0}/3]$ equidistantly.
The length of $\Gamma$ can be easily calculated:
\begin{equation}
\label{dlina}
l_\Gamma = \frac{d_0 (2\pi+2)}{3}.
\end{equation}
By the Riesz formula we have
\begin{align}
\label{e2.6}
F(\bxi)  &= - \frac{1}{2\pi i}\ointop_{\Gamma}(\A(\bxi)-\zeta I)^{-1}\, d\zeta,\ \ |\bxi|\le\delta_{0}(\alpha,\mu),
\\
\label{e2.6a}
\A(\bxi) F(\bxi) &= - \frac{1}{2\pi i}\ointop_{\Gamma}(\A(\bxi)-\zeta I)^{-1}\zeta\,d\zeta,\ \ |\bxi|\le\delta_{0}(\alpha,\mu).
\end{align}
Here we integrate along the contour $\Gamma$ in a counterclockwise direction.

Our goal is to approximate the operators $F(\bxi)$ and $\A(\bxi)F(\bxi)$  for $|\bxi|\le\delta_{0}(\alpha,\mu)$.
We begin by considering the case $0 < \alpha < 1$.
\begin{proposition}
\label{prop2.3}
Let conditions  \eqref{e1.1} and  \eqref{e1.2} be fulfilled,
and assume that $0< \alpha < 1$. Then
\begin{equation}
\label{F-P}
\| F(\bxi) - P \|_{L_2(\Omega) \to L_2(\Omega)} \le C_1(\alpha,\mu) |\bxi|^\alpha, \quad  |\bxi| \le \delta_0(\alpha,\mu)
\end{equation}
with a constant  $C_1(\alpha,\mu)$ that only depends on  $d$,
$\alpha$, $\mu_-$ and $\mu_+$.
\end{proposition}

\begin{proof}
Denote
\begin{align}
\label{res1}
R(\bxi,\zeta)&:=(\A(\bxi)-\zeta I)^{-1},\ \
|\bxi|\le\delta_{0}(\alpha,\mu),\ \ \zeta\in\Gamma;
\\
\label{res2}
R_{0}(\zeta)&:=R(\mathbf{0},\zeta),\ \ \zeta\in\Gamma.
\end{align}
Due to  formula   \eqref{e2.6} the difference $F(\bxi)-P$ admits the following
integral representation:
\begin{equation}\label{e2.30}
F(\bxi)-P= - \frac{1}{2\pi i}\ointop_{\Gamma} \left( R(\bxi,\zeta) - R_0(\zeta)\right) \, d\zeta,
\ \ |\bxi|\le\delta_{0}(\alpha,\mu).
\end{equation}
According to Lemma \ref{lem2.3}, the operator $\Delta\A(\bxi)$ is bounded if $0< \alpha < 1$. For such operators
the resolvent identity reads
\begin{equation}\label{e2.31}
R(\bxi,\zeta) = R_0(\zeta)
- R(\bxi,\zeta) \Delta\A(\bxi) R_0(\zeta), \quad
|\bxi|\le\delta_{0}(\alpha,\mu),\ \ \zeta\in\Gamma.
\end{equation}
The length of the contour $\Gamma$ is calculated in \eqref{dlina}. The contour $\Gamma$ is constructed so that
both resolvents $R(\bxi,\zeta)$ and $R_0(\zeta)$ satisfy the estimates
\begin{equation}\label{e2.32}
\| R(\bxi,\zeta)\|\le 3d_{0}^{-1},\ \
\| R_0 (\zeta)\|\le 3d_{0}^{-1}, \ \ |\bxi|\le\delta_{0}(\alpha,\mu),\quad \zeta \in \Gamma.
\end{equation}
Finally, from \eqref{e2.10}  and
\eqref{e2.30}--\eqref{e2.32} we deduce \eqref{F-P} with a constant
\begin{equation*}
\label{C1}
C_{1}(\alpha, \mu)= \frac{3(\pi+1)\mu_{+}c_{1}(d,\alpha)}{\pi d_0} =
\frac{3(\pi+1)\mu_{+}c_{1}(d,\alpha)}{\pi^{1+\alpha} \mu_- c_0(d,\alpha)} =: c'_1(d,\alpha) \frac{\mu_+}{\mu_-}.
\end{equation*}
\end{proof}

\begin{proposition}
\label{prop2.5}
Let conditions \eqref{e1.1} and \eqref{e1.2} be fulfilled and assume that  $0< \alpha < 1$.
Then the operator $\Phi(\bxi)$ defined by the relation
\begin{equation}
\label{F= P+ F1}
\A(\bxi) F(\bxi) =   P \Delta \mathbb{A}(\bxi)  P + \Phi(\bxi),
 \end{equation}
is bounded in $L_2(\mathbb R^d)$ and satisfies the estimate
\begin{equation}
\label{F-P-O(xi)}
\left\| \Phi(\bxi) \right\|_{L_2(\Omega) \to L_2(\Omega)} \le C_2(\alpha,\mu) |\bxi|^{2\alpha}, \quad  |\bxi| \le \delta_0(\alpha,\mu);
\end{equation}
the constant $C_2(\alpha,\mu)$ depends only on
$d$, $\alpha$, $\mu_-$ and $\mu_+$.
\end{proposition}

\begin{proof}
Iterating the resolvent identity from  \eqref{e2.31}, we arrive at the representation
\begin{equation}
\label{res_iden2}
R(\bxi,\zeta) = R_0(\zeta)
- R_0(\zeta) \Delta\A(\bxi) R_0(\zeta) + R(\bxi,\zeta) \Delta\A(\bxi) R_0(\zeta) \Delta\A(\bxi) R_0(\zeta), \quad
|\bxi|\le\delta_{0}(\alpha,\mu),\ \ \zeta\in\Gamma.
\end{equation}
Thanks to \eqref{e2.10} and \eqref{e2.32} the operator $Z_1(\bxi,\zeta) := R(\bxi,\zeta) \Delta\A(\bxi) R_0(\zeta) \Delta\A(\bxi) R_0(\zeta)$ can be estimated as follows:
\begin{equation}
\label{res_est1}
\| Z_1(\bxi,\zeta) \| \le (3 d_0^{-1})^3 \mu^2_+ c_1(d,\alpha)^2 |\bxi|^{2\alpha},
\quad
|\bxi|\le\delta_{0}(\alpha,\mu),\ \ \zeta\in\Gamma.
\end{equation}
Using the Riesz formula from  \eqref{e2.6a} and representation \eqref{res_iden2}, we obtain
\begin{equation*}
\label{AF=}
\A(\bxi)F(\bxi) = G_0  + G_1(\bxi) + \Phi(\bxi), \quad |\bxi|\le\delta_{0}(\alpha,\mu),
\end{equation*}
with
\begin{align}
\label{G0=}
G_0 &= - \frac{1}{2\pi i}\oint_{\Gamma} R_0(\zeta) \zeta  \, d\zeta,
\\
\label{G1=}
G_1(\bxi) &=  \frac{1}{2\pi i}\oint_{\Gamma} R_0(\zeta) \Delta \A(\bxi) R_0(\zeta) \zeta  \, d\zeta,
\\
\label{Phi=}
\Phi(\bxi) &= - \frac{1}{2\pi i}\oint_{\Gamma} Z_1(\bxi,\zeta)  \zeta  \, d\zeta.
\end{align}
Letting $\bxi=0$ in  \eqref{e2.6a} we conclude that $G_0 = \A(\mathbf{0}) P =0$.
In view of \eqref{Phi=},  estimate \eqref{F-P-O(xi)}  follows from  \eqref{res_est1},  \eqref{dlina}
and the inequality $|\zeta| \le 2d_0/3$ which holds for all $\zeta \in \Gamma$. The straightforward computation shows that
\begin{equation*}
\label{C2}
C_{2}(\alpha,\mu):=\frac{6(\pi+1) \mu_{+}^2 c_{1}(d,\alpha)^2}{\pi d_0}
=\frac{6(\pi+1) \mu_{+}^2 c_{1}(d,\alpha)^2}{\pi^{1+\alpha} \mu_-c_0(d,\alpha)}.
\end{equation*}
The calculation of the integral in \eqref{G1=} relies on the following representation for the resolvent of the operator $\mathbb{A}(\mathbf{0})$:
\begin{equation}\label{4.10}
R_{0}(\zeta)=R_{0}(\zeta)P+R_{0}(\zeta)P^{\bot}=-\frac{1}{\zeta}P+R_{0}(\zeta)P^{\bot},\ \ \zeta\in\Gamma.
\end{equation}
Substituting  the expression on the right-hand side of   \eqref{4.10} for $R_{0}(\zeta)$ in \eqref {G1=} yields
\begin{equation*}
G_{1}(\bxi) =  \frac{1}{2\pi i}\oint_{\Gamma} \Bigl( -\frac{1}{\zeta}P+R_{0}^{\bot}(\zeta)\Bigr)
\Delta \mathbb{A}(\bxi)
\Bigl(-\frac{1}{\zeta}P+R_{0}^{\bot}(\zeta)\Bigr) \zeta \,d\zeta
=  \frac{1}{2\pi i}\oint_{\Gamma} \frac{1}{\zeta}  P \Delta \mathbb{A}(\bxi)  P \,d\zeta =  P \Delta \mathbb{A}(\bxi)  P,
  \end{equation*}
  here we have also used the fact that the operator-function $R_{0}^{\bot}(\zeta):=R_{0}(\zeta)P^{\bot}$ is holomorphic
  inside the contour $\Gamma$.  This implies the desired representation  \eqref{F= P+ F1}--\eqref{F-P-O(xi)}.
\end{proof}

Now we turn to the case $1 \le  \alpha < 2$.
\begin{proposition}
\label{prop2.3a}
Let conditions  \eqref{e1.1} and \eqref{e1.2} be fulfilled,
and assume that $1 \le  \alpha < 2$.  Then the following estimate holds:
\begin{equation}
\label{F-P_a}
\| F(\bxi) - P \|_{L_2(\Omega) \to L_2(\Omega)} \le C_1(\alpha,\mu) \Theta(\bxi), \quad  |\bxi| \le \delta_0(\alpha,\mu);
\end{equation}
here $\Theta(\bxi)$  is defined in  \eqref{*.00},
and the constant $C_1(\alpha,\mu)$  depends only on $d$, $\alpha$, $\mu_-$, $\mu_+$.
\end{proposition}

\begin{proof}
If $\alpha\ge 1$, we cannot use any more relation  \eqref{e2.31} because the operator $\Delta\A(\bxi) $ need not be bounded.
Instead, we exploit the resolvent identity for operators generated by closed non-negative forms with a common domain,
see \cite[Ch. 1, \S 2]{BSu1}.
Denote by $\mathfrak{D}$ the Hilbert space $\Dom a(\mathbf{0}) = \wt{H}^\gamma(\Omega)$
equipped with the inner product
\begin{equation*}
\label{inner}
(u,v)_{\mathfrak D} :=   a(\mathbf{0})[u,v] + \mu_+ (u,v)_{L_2(\Omega)},\quad u,v \in \wt{H}^\gamma(\Omega).
\end{equation*}
Clearly, the following lower bound holds:
\begin{equation}
\label{L2.6_1}
\|u\|_{L_2(\Omega)} \le \mu_+^{-1/2}\| u \|_{\mathfrak D},\quad u \in \mathfrak{D}.
\end{equation}
The form  $a(\bxi) - a(\mathbf{0})$ is continuous in $\mathfrak D$, and thus generates a continuous self-adjoint
operator ${\mathbb T}(\bxi)$ in  $\mathfrak D$.  Then
\begin{align}
\label{L2.6_2}
a(\bxi)[u,v] - a(\mathbf{0})[u,v] = ( {\mathbb T}(\bxi) u, v)_{\mathfrak D}, \quad u,v \in {\mathfrak D},
\\
\nonumber
\|{\mathbb T}(\bxi)\|_{{\mathfrak D} \to {\mathfrak D}} = \sup_{0 \ne u \in {\mathfrak D}}
\frac{| a(\bxi)[u,u] - a(\mathbf{0})[u,u] | }{\|u\|^2_{\mathfrak D}},
\end{align}
and, by Lemma \ref{lem2.4},
\begin{equation}
\label{L2.6_4}
\|{\mathbb T}(\bxi)\|_{{\mathfrak D} \to {\mathfrak D}} \le \check{c}(d,\alpha) \Theta(\bxi).
\end{equation}
Combining this estimate with  \eqref{L2.6_1} yields
\begin{equation}
\label{L2.6_5}
\|{\mathbb T}(\bxi)\|_{{\mathfrak D} \to L_2(\Omega)} \le \mu_+^{-1/2} \check{c}(d,\alpha) \Theta(\bxi).
\end{equation}
For the resolvents of  the operators $\A(\bxi)$ and $\A(\mathbf{0})$ we keep the notation introduced in \eqref{res1} and \eqref{res2}.
According to \cite[Ch. 1, \S 2]{BSu1}, the following resolvent identity holds:
\begin{equation}
\label{L2.6_6}
R(\bxi,\zeta) - R_0(\zeta) = - \Upsilon(\zeta) {\mathbb T}(\bxi) R(\bxi,\zeta)
\end{equation}
with
\begin{equation}
\label{L2.6_7}
\Upsilon(\zeta) := I + (\zeta + \mu_+) R_0(\zeta).
\end{equation}
In order to estimate the difference of the resolvents we need an upper bound for the quantity
\begin{equation}
\label{L2.6_8}
\beta^2(\bxi) :=  \sup_{0 \ne u \in {\mathfrak D}}
\frac{ a(\mathbf{0})[u,u] + \mu_+ \|u\|^2_{L_2(\Omega)} }{a(\bxi)[u,u] + \mu_+\|u\|^2_{L_2(\Omega)}}.
\end{equation}
Recalling the definition of $ b_1(\bxi)$ in  \eqref{*.3} we have
$$
\begin{aligned}
& a(\mathbf{0})[u,u] = \frac{1}{2} \intop_{\R^d} d\y \intop_\Omega d\x \, \mu(\x,\y)
 \frac{ |u(\x) (e^{i\langle \bxi, \y-\x\rangle}-1) + (u(\x) - e^{i\langle \bxi, \y-\x\rangle} u(\y))|^2 }{|\x - \y|^{d+\alpha}}
\\
& \le  \intop_{\R^d} d\y \intop_\Omega d\x\, \mu(\x,\y)
 \frac{ |u(\x)|^2 |e^{i\langle \bxi, \y-\x\rangle}-1|^2}{|\x - \y|^{d+\alpha}}
+ \intop_{\R^d} d\y \intop_\Omega d\x \,\mu(\x,\y)
 \frac{ |u(\x) - e^{i\langle \bxi, \y-\x\rangle} u(\y)|^2 }{|\x - \y|^{d+\alpha}}
 \\
 & = 2 b_1(\bxi)[u,u] + 2 a(\bxi)[u,u],
\end{aligned}
$$
From this inequality, considering estimate \eqref{*.5} we obtain
$$
\begin{aligned}
 a(\mathbf{0})[u,u] + \mu_+ \|u\|^2_{L_2(\Omega)}  &\le
2 a(\bxi)[u,u] + \mu_+ (1+ 2 c_0(d,\alpha) |\bxi|^\alpha)\|u\|^2_{L_2(\Omega)}
\\
& \le  \max \{2, 1+ 2 c_0(d,\alpha) \pi^\alpha d^{\frac\alpha2} \}
( a(\bxi)[u,u] + \mu_+ \|u\|^2_{L_2(\Omega)} );
\end{aligned}
$$
here we also took into account the inequality $|\bxi| \le \pi \sqrt{d}$ that holds for all $\bxi \in \wt{\Omega}$.
Therefore,
 \begin{equation}
\label{L2.6_9}
\beta^2(\bxi) \le \beta_0^2(d,\alpha) := \max \{2, 1+ 2 c_0(d,\alpha) \pi^\alpha d^{\frac\alpha2} \}.
\end{equation}
As follows from \eqref{L2.6_6},
\begin{equation}
\label{L2.6_10}
\| R(\bxi,\zeta) - R_0(\zeta)\|_{L_2(\Omega) \to L_2(\Omega)} \le
\|  \Upsilon(\zeta) \|_{L_2(\Omega) \to L_2(\Omega)} \| {\mathbb T}(\bxi)\|_{\mathfrak{D} \to L_2(\Omega)}
\| R(\bxi,\zeta)\|_{L_2(\Omega) \to \mathfrak{D}}.
\end{equation}
From  \eqref{e2.32}, \eqref{L2.6_7} and the estimate $|\zeta| \le 2d_0/3$ that holds for  all $\zeta \in \Gamma$ one has
\begin{equation}
\label{L2.6_11}
\|  \Upsilon(\zeta) \|_{L_2(\Omega) \to L_2(\Omega)} \le 1 + |\zeta + \mu_+| \|R_0(\zeta)\|_{L_2(\Omega) \to L_2(\Omega)} \le 3 + 3 \mu_+ d_0^{-1}, \quad \zeta \in \Gamma.
\end{equation}
Next, we use an elementary inequality
\begin{equation}
\label{L2.6_11a}
\begin{aligned}
&\| R(\bxi,\zeta)\|_{L_2(\Omega) \to \mathfrak{D}} = \| (\A(\mathbf{0}) + \mu_+ I)^{1/2}
R(\bxi,\zeta)\|_{L_2(\Omega) \to L_2(\Omega)}
\\
&\le
\| (\A(\mathbf{0}) + \mu_+ I)^{1/2} (\A(\bxi) + \mu_+ I)^{-1/2}\|_{L_2(\Omega) \to L_2(\Omega)}
\|  (\A(\bxi) + \mu_+ I)^{1/2} R(\bxi,\zeta)\|_{L_2(\Omega) \to L_2(\Omega)}.
\end{aligned}
\end{equation}
The first term on the right-hand side here is equal to $\beta(\bxi)$, see \eqref{L2.6_8}.
Recalling the Hilbert identity
\begin{equation}
\label{L2.6_11b}
R(\bxi,\zeta) =  (\A(\bxi) + \mu_+ I)^{-1}  (I + (\mu_+ + \zeta)R(\bxi,\zeta))
\end{equation}
we can estimate the second term as follows:
\begin{equation}
\label{L2.6_11c}
\begin{aligned}
\| & (\A(\bxi) + \mu_+ I)^{1/2} R(\bxi,\zeta)\|_{L_2(\Omega) \to L_2(\Omega)} = \|  (\A(\bxi) + \mu_+ I)^{-1/2}
 (I + (\mu_+ + \zeta)R(\bxi,\zeta))\|_{L_2(\Omega) \to L_2(\Omega)}
 \\
 &\le \mu_+^{-1/2} (1 + (\mu_+ + 2d_0/3) 3 d_0^{-1}) =  \mu_+^{-1/2} (3 + 3 \mu_+ d_0^{-1}), \quad \zeta \in \Gamma,
 \quad  |\bxi| \le \delta_0(\alpha,\mu);
 \end{aligned}
\end{equation}
we have also used here \eqref{e2.32} and the fact that $|\zeta| \le 2d_0/3$ for $\zeta \in \Gamma$.
Finally, considering \eqref{L2.6_9} we conclude that
\begin{equation}
\label{L2.6_12}
\| R(\bxi,\zeta)\|_{L_2(\Omega) \to \mathfrak{D}} \le \beta_0(d,\alpha) \mu_+^{-1/2} (3 + 3 \mu_+ d_0^{-1}),
\quad \zeta \in \Gamma, \quad  |\bxi| \le \delta_0(\alpha,\mu).
\end{equation}

Combining \eqref{L2.6_5}, \eqref{L2.6_10}, \eqref{L2.6_11} and \eqref{L2.6_12} leads to the upper bound
\begin{equation}
\label{L2.6_13}
\| R(\bxi,\zeta) - R_0(\zeta)\|_{L_2(\Omega) \to L_2(\Omega)} \le
\mu_+^{-1} \check{c}(d,\alpha)\beta_0(d,\alpha)(3 + 3 \mu_+ d_0^{-1})^2 \Theta(\bxi),
\quad \zeta \in \Gamma,\quad  |\bxi| \le \delta_0(\alpha,\mu).
\end{equation}
As a consequence of this upper bound, representation \eqref{e2.30} and formula  \eqref{dlina} one has
$$
\| F(\bxi) - P\|_{L_2(\Omega) \to L_2(\Omega)} \le \frac{(\pi+1)
  \check{c}(d,\alpha)\beta_0(d,\alpha) d_0 (3 + 3 \mu_+ d_0^{-1})^2}{ 3\pi \mu_+} \Theta(\bxi),
 \quad  |\bxi| \le \delta_0(\alpha,\mu).
$$
This yields the desired estimate \eqref{F-P_a}.
\end{proof}

\begin{proposition}
\label{prop2.5_2}
Let conditions \eqref{e1.1}, \eqref{e1.2} be fulfilled, and assume that $1 \le  \alpha < 2$.
Then for  $|\bxi| \le \delta_0(\alpha,\mu)$ the following representation holds:
\begin{equation}
\label{F= P+ F1_2}
\A(\bxi) F(\bxi) =   \mu_+ P  \mathbb{T}(\bxi)  P + \wt{\Phi}(\bxi),
 \end{equation}
 where $\wt{\Phi}(\bxi)$ is a bounded operator in $L_2(\Omega)$ that satisfies the estimate
\begin{equation}
\label{F-P-O(xi)_2}
\bigl\| \wt{\Phi}(\bxi) \bigr\|_{L_2(\Omega) \to L_2(\Omega)} \le C_2(\alpha,\mu) \Theta(\bxi)^2, \quad  |\bxi| \le \delta_0(\alpha,\mu),
\end{equation}
and $\Theta(\bxi)$ is defined in  \eqref{*.00}.
The constant $C_2(\alpha,\mu)$ depends only on  $d$, $\alpha$, $\mu_-$, $\mu_+$.
\end{proposition}

\begin{proof}
An iterated resolvent identity \eqref{L2.6_6} reads
\begin{equation}
\label{L3.5_1}
R(\bxi,\zeta) = R_0(\zeta)  - \Upsilon(\zeta) {\mathbb T}(\bxi) R_0(\zeta) +
  \Upsilon(\zeta) {\mathbb T}(\bxi)  \Upsilon(\zeta) {\mathbb T}(\bxi) R(\bxi,\zeta).
\end{equation}
Let us first estimate the operator
$\wt{Z}_1(\bxi,\zeta) :=\Upsilon(\zeta) {\mathbb T}(\bxi)  \Upsilon(\zeta) {\mathbb T}(\bxi) R(\bxi,\zeta)$.
Considering  \eqref{L2.6_7} we obtain
\begin{equation*}
\label{L3.5_2}
\wt{Z}_1(\bxi,\zeta)  =
  \Upsilon(\zeta) {\mathbb T}(\bxi)^2  R(\bxi,\zeta) + (\zeta + \mu_+)  \Upsilon(\zeta) {\mathbb T}(\bxi) R_0(\zeta)
   {\mathbb T}(\bxi)  R(\bxi,\zeta).
\end{equation*}
Therefore,
\begin{equation}
\label{L3.5_3}
\begin{aligned}
\| \wt{Z}_1(\bxi,\zeta) \|_{L_2(\Omega) \to L_2(\Omega)} &\le
 \| \Upsilon(\zeta)\|_{L_2(\Omega) \to L_2(\Omega)} \| {\mathbb T}(\bxi)\|_{\mathfrak{D} \to L_2(\Omega)}
 \| R(\bxi,\zeta) \|_{L_2(\Omega) \to \mathfrak{D}}
 \\
 & \times \left( \| {\mathbb T}(\bxi)\|_{\mathfrak{D} \to {\mathfrak D}} + |\zeta + \mu_+|
  \| {\mathbb T}(\bxi)\|_{\mathfrak{D} \to L_2(\Omega)} \| R_0(\zeta) \|_{L_2(\Omega) \to \mathfrak{D}} \right).
 \end{aligned}
\end{equation}
In the same way as in \eqref{L2.6_11a}--\eqref{L2.6_11c} we derive the relation
\begin{equation}
\label{L3.5_4}
\begin{aligned}
& \| R_0(\zeta) \|_{L_2(\Omega) \to \mathfrak{D}} = \| (\A(\mathbf{0}) + \mu_+ I)^{1/2}R_0(\zeta) \|_{L_2(\Omega) \to L_2(\Omega)}
\\
&= \| (\A(\mathbf{0}) + \mu_+ I)^{-1/2} (I + (\mu_+ + \zeta) R_0(\zeta)) \|_{L_2(\Omega) \to L_2(\Omega)}
\le \mu_+^{-1/2} (3+3 \mu_+ d_0^{-1} ).
\end{aligned}
\end{equation}
Combining the inequalities in  \eqref{L2.6_4}, \eqref{L2.6_5}, \eqref{L2.6_11}, \eqref{L2.6_12}, \eqref{L3.5_3} and  \eqref{L3.5_4} we conclude that
\begin{equation}
\label{L3.5_5}
\| \wt{Z}_1(\bxi,\zeta) \|_{L_2(\Omega) \to L_2(\Omega)} \le \wt{C}_2(\alpha,\mu) \Theta(\bxi)^2, \quad \zeta \in \Gamma,
\quad |\bxi| \le \delta_0(\alpha,\mu),
\end{equation}
with a constant
$$
\wt{C}_2(\alpha,\mu) = \check{c}(d,\alpha)^2 \beta_0(d,\alpha) \mu_+^{-1} (3+ 3 \mu_+ d_0^{-1} )^2
\left(1 + (\mu_+ + 2d_0/3)  \mu_+^{-1} (3+ 3 \mu_+ d_0^{-1} )\right).
$$
Substituting representation \eqref{L3.5_1} in the Riesz formula in \eqref{e2.6a} yields
\begin{equation*}
\label{AF==}
\A(\bxi)F(\bxi) = G_0  + \wt{G}_1(\bxi) + \wt{\Phi}(\bxi), \quad |\bxi|\le\delta_{0}(\alpha,\mu),
\end{equation*}
where $G_0$ is defined in  \eqref{G0=} and two other terms on the right-hand side are given by
\begin{align}
\label{G1==}
\wt{G}_1(\bxi) &=  \frac{1}{2\pi i}\oint_{\Gamma} \Upsilon(\zeta) {\mathbb T}(\bxi) R_0(\zeta) \zeta  \, d\zeta,
\\
\label{Phi==}
\wt{\Phi}(\bxi) &= - \frac{1}{2\pi i}\oint_{\Gamma} \wt{Z}_1(\bxi,\zeta)  \zeta  \, d\zeta.
\end{align}
It has been checked that $G_0 =0$.  In view of \eqref{L3.5_5} the operator \eqref{Phi==} can be estimated as follows:
\begin{equation*}
\label{L3.5_6}
\| \wt{\Phi}(\bxi) \|_{L_2(\Omega) \to L_2(\Omega)} \le \frac{2 (\pi+1) d_0^2}{9\pi}\wt{C}_2(\alpha,\mu) \Theta(\bxi)^2
=: C_2(\alpha,\mu) \Theta(\bxi)^2,
\quad |\bxi| \le \delta_0(\alpha,\mu).
\end{equation*}
This yields  \eqref{F-P-O(xi)_2}.

It remains to substitute representation \eqref{4.10} into the contour integral in  \eqref{G1==}. Since
the operator-function  $R_{0}^{\bot}(\zeta)$ is holomorphic inside the contour  $\Gamma$, we have
\begin{equation*}
\begin{aligned}
\wt{G}_{1}(\bxi) &=  \frac{1}{2\pi i}\oint_{\Gamma} \Bigl( I + (\zeta + \mu_+) \Bigl(-\frac{1}{\zeta}P+R_{0}^{\bot}(\zeta)\Bigr) \Bigr) \mathbb{T}(\bxi)
\Bigl(-\frac{1}{\zeta}P+R_{0}^{\bot}(\zeta)\Bigr) \zeta \,d\zeta
\\
&=  \frac{1}{2\pi i}\oint_{\Gamma} \frac{1}{\zeta} \mu_+ P \mathbb{T}(\bxi)  P \,d\zeta =
 \mu_+ P \mathbb{T}(\bxi)  P.
\end{aligned}
  \end{equation*}
This implies the representation in    \eqref{F= P+ F1_2}.
\end{proof}

\begin{remark}
\label{rem3.6}
 By the definition of the operator ${\mathbb T}(\bxi)$ in  \eqref{L2.6_2}, we have
$$
a(\bxi)[Pu,Pv] - a(\mathbf{0})[Pu,Pv] = a(\mathbf{0})[{\mathbb T}(\bxi)Pu,Pv] + \mu_+ ({\mathbb T}(\bxi)Pu,Pv)_{L_2(\Omega)}.
$$
Since $P$ is the orthogonal projection onto  $\Ker \A(\mathbf{0})$, then $a(\mathbf{0}) [{\mathbb T}(\bxi)P u, Pv] =0$,
and,  thus
\begin{equation}
\label{remm1}
a(\bxi)[Pu,Pv] - a(\mathbf{0})[Pu,Pv] = \mu_+ ( {\mathbb T}(\bxi) Pu, Pv)_{L_2(\Omega)}, \quad u,v \in L_2(\Omega).
  \end{equation}
  Therefore, $\mu_+ P {\mathbb T}(\bxi) P$ is a bounded operator in  $L_2(\Omega)$ that corresponds to the form
  in the left-hand side of \eqref{remm1}.
For consistency of notation in the cases $1 \le \alpha<2$ and $0< \alpha <1$, in what follows we denote
  $$
  \mu_+ P {\mathbb T}(\bxi) P = P \Delta \A(\bxi) P.
  $$
\end{remark}

\subsection{Analysis of operator $ P \Delta \mathbb{A}(\bxi)  P$}

By the definition of the operator $\Delta \A(\bxi)$, taking into account \eqref{remm1} and the relation $P = (\cdot, \mathbf{1}_\Omega)  \mathbf{1}_\Omega$, we have
\begin{equation}\label{2.30}
 P \Delta \mathbb{A}(\bxi)  P = \rho(\bxi) P,
  \end{equation}
where
\begin{equation}\label{2.31}
\begin{aligned}
\rho(\bxi) = a(\bxi)[{\mathbf 1}_\Omega,{\mathbf 1}_\Omega] - a(\mathbf{0})[{\mathbf 1}_\Omega,{\mathbf 1}_\Omega] =a(\bxi)[{\mathbf 1}_\Omega,{\mathbf 1}_\Omega]=
\intop_{\R^d} d\y \intop_\Omega d\x\, \mu(\x,\y) \frac{ 1 - \cos(\langle \bxi, \x -\y\rangle)}{|\x -\y |^{d+\alpha}}.
 \end{aligned}
  \end{equation}
The Fourier series of the periodic function $\mu(\x,\y)$ reads
\begin{equation}\label{2.32}
\mu(\x,\y) = \sum_{\m, \mathbf{l} \in \Z^d} \wh{\mu}_{\m, \mathbf{l}} e^{2\pi i (\langle \m, \x\rangle + \langle \mathbf{l}, \y\rangle)}
  \end{equation}
with the coefficients    $\wh{\mu}_{\m, \mathbf{l}}$ defined by
  $$
  \wh{\mu}_{\m, \mathbf{l}} = \intop_\Omega \intop_\Omega d\x\, d\y\, \mu(\x,\y) e^{-2\pi i (\langle \m, \x\rangle + \langle \mathbf{l}, \y\rangle)},\quad \m, \mathbf{l} \in \Z^d.
  $$
  The symmetry of $\mu(\x,\y)$ implies that
$ \wh{\mu}_{\m, \mathbf{l}} =  \wh{\mu}_{ \mathbf{l}, \m}$, $\m, \mathbf{l} \in \Z^d$.
Denote
\begin{equation}\label{2.33}
\mu^0  :=  \wh{\mu}_{\mathbf{0}, \mathbf{0}} = \intop_\Omega \intop_\Omega d\x\, d\y\, \mu(\x,\y).
  \end{equation}

\begin{lemma}
\label{lem2.6}
The function $\rho(\bxi)$ given by  \eqref{2.31} admits the representation
\begin{align}
\label{2.35}
\rho(\bxi)& = \mu^0 V_\alpha(\bxi) +  \rho_*(\bxi),
\\
\label{2.35aa}
\rho_*(\bxi) &= \intop_{\R^d} d\z \,  \mu_*(\z)
 \frac{ 1 - \cos (\langle \bxi, \z\rangle)}{|\z |^{d+\alpha}},
 \end{align}
 where $\mu^0$ and $V_\alpha(\bxi)$ are introduced in \eqref{2.33} and  \eqref{e1.6c}, respectively, and
 $\mu_*$ is a zero average, even,  $\Z^d$-periodic function given by
$\mu_*(\z) = \int_\Omega \mu(\x,\z+\x) \,d\x - \mu^0$. Furthermore, the function $\mu_*$
satisfies the estimates
\begin{equation}
\label{2.37}
 \mu_- - \mu^0 \le \mu_*(\z) \le \mu_+ - \mu^0,
  \end{equation}
  and its Fourier series reads
  \begin{equation}
\label{2.34}
\mu_*(\z) = \sum_{\m \in \Z^d \setminus \{ \mathbf{0} \} } \wh{\mu}_{\m,-\m}  \cos (2\pi  \langle\m,\z\rangle ).
\end{equation}
\end{lemma}

\begin{proof}
Making the change of variables $\y \mapsto \z = \y - \x$ in  \eqref{2.31} yields
\begin{equation}
\label{2.35a}
\rho(\bxi) =
  \intop_\Omega d\x \intop_{\R^d} d\z \, \mu(\x,\z+\x)
 \frac{ 1 - \cos (\langle \bxi, \z\rangle)}{|\z |^{d+\alpha}} =
  \intop_{\R^d} d\z\, f(\z)  \frac{ 1 - \cos (\langle \bxi, \z\rangle)}{|\z |^{d+\alpha}},
\end{equation}
where
%где
$$
f(\z) =  \intop_{\Omega} d\x \, \mu(\x,\z+\x).
$$
Notice that due to conditions \eqref{e1.1} and \eqref{e1.2}, the function $f(\z)$ is $\Z^d$-periodic
and such that $f(\z)= f(-\z)$. Moreover, $f(\z)$ satisfies the estimates
$\mu_- \le f(\z) \le \mu_+$.

Taking into account the periodicity of $\mu$ one can calculate the average of $f(\z)$:
$$
\intop_\Omega f(\z) \, d\z =  \intop_\Omega \intop_{\Omega} d\z\,d\x \, \mu(\x,\z+\x) =
 \intop_\Omega \intop_{\Omega} d\y\,d\x \, \mu(\x,\y) = \mu^0.
$$
Letting
\begin{equation}
\label{2.36}
\mu_*(\z) := f(\z) - \mu^0,
  \end{equation}
from \eqref{2.35a} and the definition of  $V_\alpha(\bxi)$ in  \eqref{e1.6b} we deduce \eqref{2.35}--\eqref{2.35aa}.

By construction, $\mu_*(\z)$ is a $\Z^d$-periodic, even, zero average function that satisfies estimates    \eqref{2.37}.
Denote by $\wh{\mu}_{*,\m}$, $\m\in\mathbb Z^d$, the Fourier coefficients of the function  $\mu_*$.
Since the average of  $\mu_*$ is equal to zero, $\wh{\mu}_{*,0}=0$. For
\hbox{$\m \in \Z^d \setminus \{ \mathbf{0} \}$} we have
$$
\wh{\mu}_{*,\m} = \intop_\Omega \intop_\Omega d\z\, d\x\, (\mu(\x,\z+\x) - \mu^0) e^{- 2\pi i \langle\m,\z\rangle}=
\intop_\Omega \intop_\Omega d\x\, d\y\, \mu(\x,\y)  e^{- 2\pi i \langle\m,\y - \x\rangle}
= \wh{\mu}_{-\m,\m}.
$$
Therefore,
$$
\mu_*(\z) = \sum_{\m \in \Z^d \setminus \{ \mathbf{0} \} } \wh{\mu}_{-\m,\m}  e^{ 2\pi i \langle\m,\z\rangle}.
$$
Taking into account the symmetry $\wh{\mu}_{\m,-\m} = \wh{\mu}_{-\m,\m}$ we finally obtain \eqref{2.34}.
\end{proof}

\begin{lemma}
\label{lem2.7}
Let conditions \eqref{e1.1} and \eqref{e1.2} be satisfied, and assume that $0< \alpha< 2$. Then
for the function $\rho_*(\bxi)$ defined in \eqref{2.35aa}  the following upper bound holds true:
\begin{equation}
\label{2.40}
|\rho_*(\bxi) | \leq C_3(\alpha,\mu) \begin{cases}|\bxi|^{1+\alpha}, & \hbox{if } 0 < \alpha < 1, \\
|\bxi|^2  (1 + |\operatorname{ln} |\bxi||), &  \hbox{if } \alpha =1, \\ |\bxi|^2, & \hbox{if } 1 < \alpha < 2,\end{cases}\qquad
\hbox{for all }\bxi \in \wt{\Omega}.
  \end{equation}
  The constant $C_3(\alpha,\mu)$  depends only on  $d$, $\alpha$ and  $\mu_+$.
\end{lemma}

\begin{proof}
Choose $R = R(d,\alpha)>0$ such that $\left(\frac{R}{R - \sqrt{d}}\right)^{d+\alpha +1} = 2$. An easy computation
yields
$$
R(d,\alpha) = \frac{\sqrt{d} \,2^{\frac{1}{d+\alpha+1}} }{2^{\frac{1}{d+\alpha +1}}-1}.
$$
Let $\Sigma$ be a collection of indices  $\n \in \Z^d$ such that  $\Omega + \n$
has a non-trivial intersection with the set $\{\z \in \R^d: R \le |\z| \le R \pi \sqrt{d}|\bxi|^{-1}\}$,
and consider the following partition of $\mathbb R^d$: \ $\mathbb R^d=\Xi_1\cup\Xi_2\cup\Xi_3$ with
$$
\Xi_1 = \bigcup_{\n \in \Sigma} (\Omega + \n), \quad \Xi_2 =( \R^d \setminus \Xi_1) \cap \{\z: |\z| < R\},
\quad \Xi_3 =( \R^d \setminus \Xi_1) \cap \{\z: |\z| > R\pi \sqrt{d}|\bxi|^{-1} \}.
$$
By the choice of  $R$ we have
\begin{equation}
\label{2.41}
 \frac{|\z_1|^{d+\alpha -1}}{|\z_2|^{d+\alpha -1}} \le 2,\quad\hbox{if } \z_1,\z_2 \in \Omega + \n, \ \ \hbox{and }\ \n \in \Sigma.
  \end{equation}
  Clearly,
\begin{equation}
\label{2.42}
\rho_*(\bxi) = \rho_*^{(1)} (\bxi) + \rho_*^{(2)} (\bxi) + \rho_*^{(3)} (\bxi)
  \end{equation}
with
$$
\rho_*^{(j)} (\bxi) = \intop_{\Xi_j} d\z \,  \mu_*(\z)
 \frac{ 1 - \cos (\langle \bxi, \z\rangle)}{|\z |^{d+\alpha}},\quad  j=1,2,3.
 $$
In order to estimate the function $\rho_*^{(2)} (\bxi)$ we use the upper bound $|\mu_*(\x)| \le \mu_+$ which is a straightforward consequence of  \eqref{2.37}, and an evident inequality  $1 - \cos (\langle \bxi, \z\rangle) = 2 \sin^2 (\frac{\langle \bxi, \z\rangle}{2}) \le \frac{|\bxi|^2 |\z|^2}{2}$. This yields
\begin{equation}
\label{2.43}
|\rho_*^{(2)}(\bxi)| \le \frac{1}{2}\mu_+  |\bxi|^2 \intop_{|\z| <R} \frac{d\z}{ |\z|^{d + \alpha -2}} =
   C_3^{(2)} |\bxi|^2,  \quad
  C_3^{(2)} = \frac{\mu_+  \omega_d  R^{2-\alpha}}{2(2-\alpha)}.
  \end{equation}

Denote
$$
\varphi_{\bxi} (\z) := \frac{ 1 - \cos (\langle \bxi, \z\rangle)}{|\z |^{d+\alpha}};
\quad  \varphi_{\bxi}^{(\n)} := \intop_{\Omega + \n} \varphi_{\bxi} (\z)\, d\z,\quad \n\in\Sigma.
$$
Since the function  $\mu_*$ is periodic and $\int_\Omega \mu_*(\z)\,d\z =0$, we have
$$
\intop_{\Omega + \n} d\z \,  \mu_*(\z) \varphi_{\bxi} (\z) =
\intop_{\Omega + \n} d\z \,  \mu_*(\z) \bigl(\varphi_{\bxi} (\z) - \varphi_{\bxi}^{(\n)}).
$$
From the definition of $\varphi_{\bxi}$ it follows that
$$
|\nabla \varphi_{\bxi} (\z)| \le \frac{|\bxi| |\sin (\langle \bxi, \z\rangle)|}{|\z|^{d+\alpha}}
+ (d+\alpha) \frac{ 1 - \cos (\langle \bxi, \z\rangle)}{|\z|^{d+\alpha +1}}
\le \frac{ (d+2+\alpha) |\bxi|^2}{ 2 |\z|^{d+\alpha-1}}.
$$
This in turn  implies the inequality
$$
\max_{\z \in \Omega + \n} | \varphi_{\bxi} (\z) - \varphi_{\bxi}^{(\n)}| \le \sqrt{d} \max_{\z \in \Omega + \n} |\nabla \varphi_{\bxi} (\z)| \le \sqrt{d} \max_{\z \in \Omega + \n} \frac{ (d+2+\alpha) |\bxi|^2}{ 2 |\z|^{d+\alpha-1}}.
$$
Therefore,
$$
\left| \int_{\Omega + \n} d\z \,  \mu_*(\z) \varphi_{\bxi} (\z) \right| \le
\mu_+ \sqrt{d} \max_{\z \in \Omega + \n} \frac{ (d+2+\alpha) |\bxi|^2}{ 2 |\z|^{d+\alpha-1}}
\le  \mu_+ \sqrt{d} (d+ 2+ \alpha) |\bxi|^2 \intop_{\Omega + \n}   \frac{d\z}{|\z|^{d+\alpha -1}};
$$
in the last inequality here \eqref{2.41} has been used.  Summing these estimates over $\n \in \Sigma$ yields
\begin{equation}
\label{2.44}
\begin{aligned}
|\rho_*^{(1)} (\bxi)| \le \sqrt{d} (d+2+\alpha) \mu_+ |\bxi|^2\intop_{R - \sqrt{d} \le |\z| \le R \pi \sqrt{d}|\bxi|^{-1}+ \sqrt{d} }   \frac{d\z}{|\z|^{d+\alpha-1}}.
\end{aligned}
\end{equation}
Calculating the integral on the right-hand side  we have
\begin{equation}
\label{2.44a}
\begin{aligned}
 |\rho_*^{(1)} (\bxi)|
& \le C_3^{(1)} \begin{cases} |\bxi|^{1+\alpha}, &\hbox{if } 0 < \alpha < 1, \\
 |\bxi|^2 (1 + |\operatorname{ln} |\bxi||), &\hbox{if } \alpha =1,
 \\  |\bxi|^2,  & \hbox{if }  1 < \alpha < 2,\end{cases}
\end{aligned}
\end{equation}
with
$$
\begin{aligned}
C_3^{(1)} = C_3^{(1)}(d,\alpha) &= \sqrt{d} (d+2+\alpha) \mu_+  \omega_d
\begin{cases}  (1-\alpha)^{-1} (R \pi \sqrt{d} + d\pi)^{1-\alpha} , &\hbox{if } 0 < \alpha < 1,
\\ \operatorname{ln} (R \pi \sqrt{d} + d\pi), &\hbox{if }  \alpha=1,
\\  (\alpha -1)^{-1} (R - \sqrt{d})^{1-\alpha}, &\hbox{if }  1 < \alpha < 2. \end{cases}
\end{aligned}
$$

Considering the estimate
$$
|\nabla \varphi_{\bxi} (\z)| \le \frac{|\bxi| |\sin (\langle \bxi, \z\rangle)|}{|\z|^{d+\alpha}}
+ (d+\alpha) \frac{ 1 - \cos (\langle \bxi, \z\rangle)}{|\z|^{d+\alpha +1}}
\le \frac{|\bxi| }{|\z|^{d+\alpha}}
+  \frac{ 2(d+\alpha)}{|\z|^{d+\alpha +1}},
$$
in the same way as in the proof of inequality \eqref{2.44} we obtain
\begin{equation}
\label{2.45}
\begin{aligned}
|\rho_*^{(3)} (\bxi)| \le 2 \mu_+ \sqrt{d} \intop_{|\z| > R \pi \sqrt{d}|\bxi|^{-1} - \sqrt{d} } d\z\,
 \left( \frac{|\bxi| }{|\z|^{d+\alpha}}+  \frac{ 2(d+\alpha)}{|\z|^{d+\alpha +1}} \right)
\le C_3^{(3)}|\bxi|^{1+\alpha},
\\
C_3^{(3)} = 2 \mu_+ \sqrt{d}\, \omega_d \left(  \frac{(R \pi \sqrt{d} - d\pi)^{-\alpha} }{\alpha}
+ \frac{ 2(d+\alpha) (R \pi \sqrt{d} - d\pi)^{-1-\alpha} }{1+\alpha} \right).
\end{aligned}
\end{equation}
Finally, combining \eqref{2.42}, \eqref{2.43}, \eqref{2.44a}, \eqref{2.45}  and taking into account the inequality
$|\bxi| \le \pi \sqrt{d}$,  $\bxi \in \wt{\Omega}$, we arrive at the desired estimate \eqref{2.40} with the constant
\begin{equation*}
\label{2.46}
C_3(\alpha,\mu) = \begin{cases} C_3^{(1)}
+ C_3^{(2)}(\pi \sqrt{d})^{1-\alpha} + C_3^{(3)}, &\hbox{if } 0 < \alpha < 1,
\\
C_3^{(1)} +C_3^{(2)} + C_3^{(3)}, &\hbox{if }   \alpha =1,
\\
C_3^{(1)} +C_3^{(2)} + C_3^{(3)}(\pi \sqrt{d})^{\alpha-1}, &\hbox{if }  1 < \alpha < 2.
\end{cases}
\end{equation*}
\end{proof}

The following statement is a consequence of Propositions  \ref{prop2.5} and \ref{prop2.5_2}, relation  \eqref{2.30}
and Lemmas \ref{lem2.6} and \ref{lem2.7}:

\begin{proposition}
\label{prop2.8}
Let conditions \eqref{e1.1} and  \eqref{e1.2} be fulfilled, and assume that   $0< \alpha < 2$. Then for all $\bxi$ such that
$|\bxi| \le \delta_0(\alpha,\mu)$  the following estimate holds:
\begin{equation*}
\label{2.47}
\left\| \A(\bxi) F(\bxi) - \mu^0 V_\alpha (\bxi) P \right\|_{L_2(\Omega) \to L_2(\Omega)} \le C_4(\alpha,\mu) \begin{cases}
|\bxi|^{2\alpha}, &\hbox{if }   0 < \alpha < 1, \\
|\bxi|^2  (1 + |\operatorname{ln} |\bxi||)^2, &\hbox{if }  \alpha =1, \\
|\bxi|^2, &\hbox{if }  1< \alpha <2,
\end{cases}
\end{equation*}
here $\mu^0$ and $V_\alpha(\bxi)$ are defined in  \eqref{2.33} and \eqref{e1.6c}, respectively, and the constant
 $C_4(\alpha,\mu)$  is given by the expression
$$
C_4(\alpha,\mu) = \begin{cases} C_2(\alpha,\mu) + C_3(\alpha,\mu) (\pi \sqrt{d})^{1-\alpha}, &\hbox{if }
 0< \alpha <1,
\\
C_2(\alpha,\mu) + C_3(\alpha,\mu), &\hbox{if }  1 \le \alpha < 2,
\end{cases}
$$
this constant depends only on $d$, $\alpha$, $\mu_-$, $\mu_+$.
\end{proposition}

\section{Approximation of the resolvent  $(\A + \eps^\alpha I)^{-1}$}\label{Sec3}

\subsection{Approximation of the resolvent  of operator  $\A(\bxi)$}\label{sec3.1}

Denote
\begin{equation}
\label{4.1}
\Theta(\bxi) := \begin{cases}
|\bxi|^{\alpha}, &\hbox{if }   0 < \alpha < 1, \\
|\bxi| (1 + |\operatorname{ln} |\bxi||), &\hbox{if }  \alpha =1, \\
|\bxi|, &\hbox{if }  1< \alpha <2.
\end{cases}
\end{equation}
This notation is compatible with that in \eqref{*.00}. We also define the operator
\begin{equation*}
\label{Xi=}
\Xi(\bxi,\eps) := (\A(\bxi)+\varepsilon^{\alpha}I)^{-1} F(\bxi) - \left( \mu^0 V_\alpha(\bxi) + \eps^\alpha \right)^{-1} P,
 \ \ \bxi\in\widetilde{\Omega}, \ \ \eps >0,
\end{equation*}
where $\mu^0$ is given by \eqref{2.33}, and $V_\alpha(\bxi)$ by \eqref{e1.6c}.
\begin{proposition}
Let conditions \eqref{e1.1} and \eqref{e1.2} be satisfied, and assume that $0 < \alpha < 2$.
Then the following estimate holds:
\begin{equation}
\label{Xi=le}
\| \Xi(\bxi,\eps) \|_{L_2(\Omega) \to L_2(\Omega)} \le
\frac{2 C_1(\alpha,\mu) \Theta(\bxi) }{\mu_{-}c_0(d,\alpha)|\bxi|^{\alpha} + \eps^\alpha} +
\frac{C_4(\alpha,\mu) \Theta(\bxi)^2 }{(\mu_{-}c_0(d,\alpha) |\bxi|^{\alpha} + \eps^\alpha)^2},
\quad |\bxi|\le\delta_{0}(\alpha,\mu), \quad \eps>0,
\end{equation}
with $\Theta(\bxi)$ defined in \eqref{4.1}.
\end{proposition}

\begin{proof}
According to  \eqref{e1.6c} and Proposition\ref{prop1.5} we have
\begin{align}\label{e2.39}
\|(\A(\bxi)+\varepsilon^{\alpha}I)^{-1} \|
& \le(\mu_{-} c_0(d,\alpha)|\bxi|^{\alpha}+\varepsilon^{\alpha})^{-1},\
\ \varepsilon>0,\ \ \bxi \in \wt{\Omega};
\\
\label{e2.40}
(\mu^0 V_\alpha ( \bxi)  +\varepsilon^{\alpha})^{-1}
& \le(\mu_{-} c_0(d,\alpha) |\bxi|^{\alpha}+\varepsilon^{\alpha})^{-1},\
\ \varepsilon>0,\ \ \bxi \in \wt{\Omega}.
\end{align}
After simple rearrangements the formula for $\Xi(\bxi,\eps)$ takes the form
\begin{multline}\label{e2.41}
\Xi(\bxi,\eps) =
F(\bxi)(\A(\bxi)+\varepsilon^{\alpha}I)^{-1}(F(\bxi)-P)+(F(\bxi)-P) (\mu^0 V_\alpha(\bxi) +\varepsilon^{\alpha})^{-1} P
-
\\
-
F(\bxi)(\A(\bxi)+\varepsilon^{\alpha}I)^{-1}
\left(\A(\bxi)F(\bxi) - \mu^0 V_\alpha(\bxi) P \right)
( \mu^0 V_\alpha(\bxi) +\varepsilon^{\alpha} )^{-1} P,
\end{multline}
and the required inequality \eqref{Xi=le} follows from Propositions \ref{prop2.3},  \ref{prop2.3a}, \ref{prop2.8}
and from relations (\ref{e2.39})--(\ref{e2.41}).
\end{proof}

\begin{theorem}\label{teor2.2}
Let conditions \eqref{e1.1} and \eqref{e1.2} be fulfilled, and assume that $0 < \alpha < 2$.
Then for all $\varepsilon>0$ and all $\bxi$ such that  \hbox{$|\bxi|\le\delta_{0}(\alpha,\mu)$} we have
\begin{equation}\label{e2.37}
\left\|(\A(\bxi)+\varepsilon^{\alpha}I)^{-1}\!-
(\mu^0 V_\alpha(\bxi) +\varepsilon^{\alpha})^{-1} P \right\|_{L_2(\Omega) \to L_2(\Omega)}\! \le
C_5(\alpha,\mu)\! \begin{cases} 1, &\hbox{if }  0 < \alpha <1,\\
(1 + | \operatorname{ln} \eps|)^2, &\hbox{if }  \alpha =1,
\\ \eps^{2 - 2\alpha}, &\hbox{if }  1< \alpha < 2,
\end{cases}
\end{equation}
where $\mu^0$ is given by \eqref{2.33}, and  $V_\alpha(\bxi)$ by \eqref{e1.6c}.
The constant $C_5(\alpha,\mu)$  depends only on
$d$, $\alpha$, $\mu_-$, $\mu_+$.
\end{theorem}

\begin{proof}
It follows from the definition of  $F(\bxi)$ and Proposition \ref{prop2.1} that
\begin{equation}\label{e2.38}
\|(\A(\bxi)+\varepsilon^{\alpha}I)^{-1}(I-F(\bxi))\|\le
\frac{1}{d_{0}},\ \ |\bxi|\le\delta_{0}(\alpha,\mu),\ \ \eps >0.
\end{equation}
In the case $0< \alpha < 1$  inequality \eqref{Xi=le} implies the estimate
\begin{equation}
\label{Xi=le_2}
\| \Xi(\bxi,\eps) \| \le
 \frac{2 C_{1}(\alpha,\mu)}{\mu_{-}c_0(d,\alpha)}+\frac{C_{4}(\alpha,\mu)}{(\mu_{-}c_0(d,\alpha))^{2}} =:
 \wt{C}_5(\alpha,\mu),\quad |\bxi|\le\delta_{0}(\alpha,\mu), \quad \eps>0, \quad 0 < \alpha < 1.
\end{equation}
Since
$$
(\A(\bxi)+\varepsilon^{\alpha}I)^{-1}-(\mu^0 V_\alpha(\bxi) +\varepsilon^{\alpha})^{-1} P=
(\A(\bxi)+\varepsilon^{\alpha}I)^{-1}(I-F(\bxi)) +\Xi(\bxi,\eps),
$$
by \eqref{e2.38} and \eqref{Xi=le_2} we obtain \eqref{e2.37} with  constant
$C_5(\alpha,\mu)= d_0^{-1} + \wt{C}_5(\alpha,\mu) $.

In the case $1< \alpha <2$ it follows from  \eqref{Xi=le} that
\begin{equation*}
\label{Xi=le_3}
\begin{aligned}
\| \Xi(\bxi,\eps) \| &\le
 |\bxi|^{\alpha -1} \cdot \frac{2 C_{1}(\alpha,\mu)|\bxi|^{2-\alpha}}{(\mu_{-}c_0(d,\alpha) |\bxi|^\alpha + \eps^\alpha)^{2/\alpha -1} } \cdot \frac{1}{(\mu_{-}c_0(d,\alpha) |\bxi|^{\alpha} + \eps^\alpha)^{2-2/\alpha}}
 \\
 &+\frac{C_{4}(\alpha,\mu)|\bxi|^{2}}{(\mu_{-}c_0(d,\alpha) |\bxi|^{\alpha} + \eps^\alpha)^{2/\alpha}} \cdot
 \frac{1}{(\mu_{-}c_0(d,\alpha) |\bxi|^{\alpha} + \eps^\alpha)^{2-2/\alpha}}
 \\
 &\le \Bigl( \frac{2 C_{1}(\alpha,\mu) (\pi \sqrt{d})^{\alpha -1}}{(\mu_{-}c_0(d,\alpha))^{2/\alpha -1}}
 + \frac{ C_{4}(\alpha,\mu) }{(\mu_{-}c_0(d,\alpha))^{2/\alpha}}  \Bigr)
\eps^{2 - 2\alpha} =: \wt{C}_5(\alpha,\mu)\eps^{2 - 2\alpha},
\\
& \ \ \ \ \ |\bxi|\le\delta_{0}(\alpha,\mu), \quad \eps>0, \quad 1 < \alpha < 2.
 \end{aligned}
\end{equation*}
Combining this inequality with \eqref{e2.38} yields \eqref{e2.37} with constant
$C_5(\alpha,\mu)= d_0^{1-2/\alpha} + \wt{C}_5(\alpha,\mu) $. % (in the case $1< \alpha <2$).

In the remaining case $ \alpha =1$ as a consequence of  \eqref{Xi=le} we have
\begin{equation}
\label{Xi=le_4}
\begin{aligned}
\| \Xi(\bxi,\eps) \| &\le
 \frac{2 C_{1}(1,\mu) |\bxi| (1 + | \operatorname{ln} |\bxi| |)}{\mu_{-}c_0(d,1) |\bxi| + \eps }
 +\frac{C_{4}(1,\mu) |\bxi|^{2} (1 + | \operatorname{ln} |\bxi| |)^2 }{(\mu_{-}c_0(d,1) |\bxi| + \eps)^{2}}
  \\
 &\le  \wt{C}_5(1,\mu)  (1 + | \operatorname{ln} \eps |)^2 ,
\quad |\bxi|\le\delta_{0}(\alpha,\mu), \quad \eps>0, \quad \alpha  =1.
 \end{aligned}
\end{equation}
Here we have used the inequality
$$
\sup_{0< s \le \delta_0}   \frac{s (1 +|\operatorname{ln} s |)}{c s + \eps } \le C_{\operatorname{max}}(c)(1 +|\operatorname{ln} \eps|), \quad \eps >0,
$$
that can be easily justified by calculating the maximum of the function $f(s) =
 \frac{s (1 +|\operatorname{ln} s |)}{c s + \eps }$.
The constant in \eqref{Xi=le_4} is given by the expression
$$
\wt{C}_5(1,\mu) = 2 C_{1}(1,\mu) C_{\operatorname{max}}( \mu_{-}c_0(d,1) ) +
C_{4}(1,\mu) C_{\operatorname{max}}^2( \mu_{-}c_0(d,1) ).
$$
From    \eqref{e2.38} and  \eqref{Xi=le_4} we obtain estimate \eqref{e2.37} with a constant
$C_5(1,\mu)= d_0^{-1} + \wt{C}_5(1,\mu) $.
\end{proof}

The {\sl effective operator} $\A^{0}$ is introduced as a self-adjoint operator in $L_2(\mathbb R^d)$ which is
generated by form \eqref{e1.3} with a constant coefficient $\mu^0$ given by \eqref{2.33}.
It is well-known that
\begin{equation}
\label{eff_op}
\A^{0}:= \A(\alpha,\mu^0) = \mu^0 \A_0(\alpha) = \mu^0 c_0(d,\alpha) (- \Delta)^\gamma,
\quad \operatorname{Dom} \A^0 = H^{\alpha}(\R^d).
\end{equation}

With the help of the unitary Gelfand transform the operator  $\A^0$ can be expanded into a direct integral as follows:
\begin{equation}
\label{direct_int}
\A^0 = {\mathcal G}^*\Bigl( \int_{\wt{\Omega}} \oplus \A^0(\bxi)\,d\bxi  \Bigr) {\mathcal G};
\end{equation}
here
$$
\A^0(\bxi) = \A(\bxi;\alpha,\mu^0) = \mu^0 \A_0(\bxi;\alpha) = \mu^0 c_0(d,\alpha) | \D + \bxi|^{\alpha},  \quad \Dom{\mathbb A}^0(\bxi) =  \wt{H}^{\alpha}(\Omega);
$$
see  \eqref{e1.11} and \eqref{e1.14} for the details.

As a direct consequence of Theorem \ref{teor2.2} we have
\begin{theorem}\label{teor3.3}
Let conditions \eqref{e1.1} and \eqref{e1.2} be fulfilled, and assume that $0 < \alpha < 2$.
Then for all $\varepsilon>0$ and $\bxi \in \wt{\Omega}$ the following estimate holds:
\begin{equation}\label{e2.37_2}
\left\|(\A(\bxi)+\varepsilon^{\alpha}I)^{-1}- (\A^0(\bxi)+\varepsilon^{\alpha}I)^{-1}
\right\|_{L_2(\Omega) \to L_2(\Omega)}\le
{\mathrm C}(\alpha,\mu)  \begin{cases} 1, &\hbox{if }  0 < \alpha <1,\\
(1 + | \operatorname{ln} \eps|)^2, &\hbox{if }  \alpha =1,
\\ \eps^{2 - 2\alpha}, &\hbox{if }  1< \alpha < 2;
\end{cases}
\end{equation}
here the constant ${\mathrm C}(\alpha,\mu)$ depends only
on $d$, $\alpha$, $\mu_-$, $\mu_+$.
\end{theorem}

\begin{proof}
As a strightforward consequence of \eqref{e2.39} and  \eqref{e2.40} we obtain
$$
\begin{aligned}
\left\|(\A(\bxi)+\varepsilon^{\alpha}I)^{-1} \right\| &\le (\mu_- c_0(d,\alpha)\delta_0(\alpha,\mu)^\alpha)^{-1},
\quad \eps>0, \quad \bxi \in \wt{\Omega}, \quad |\bxi| \ge \delta_0(\alpha,\mu),
\\
 \left( \mu^0 V_\alpha (\bxi) + \eps^\alpha \right)^{-1} & \le (\mu_- c_0(d,\alpha)\delta_0(\alpha,\mu)^\alpha)^{-1},
 \quad \eps >0,\quad \bxi \in \wt{\Omega}, \quad |\bxi| \ge \delta_0(\alpha,\mu).
\end{aligned}
$$
Combining these inequalities and estimate \eqref{e2.37} of Theorem \ref{teor2.2}, we conclude that
\begin{equation}
\label{e2.37_all}
\left\|(\A(\bxi)+\varepsilon^{\alpha}I)^{-1}-
( \mu^0 V_\alpha(\bxi) +\varepsilon^{\alpha})^{-1} P \right\|\le
\check{C}_5(\alpha,\mu)
\begin{cases} 1, &\hbox{if }  0 < \alpha <1,\\
(1 + | \operatorname{ln} \eps|)^2, &\hbox{if }  \alpha =1,
\\ \eps^{2 - 2\alpha}, &\hbox{if }  1< \alpha < 2,
\end{cases}
\end{equation}
for all $\varepsilon>0$ and $\bxi \in \wt{\Omega}$;
the constant $\check{C}_5(\alpha,\mu)$ depends on whether  $0 < \alpha \le 1$ or $1< \alpha <2$:
$$
\check{C}_5(\alpha,\mu) =
\left\{
\begin{array}{ll}
\displaystyle
\max \{ C_5(\alpha,\mu), 2 (\mu_- c_0(d,\alpha)\delta_0(\alpha,\mu)^\alpha)^{-1} \},&\hbox{if }0 < \alpha \le 1,\\[1.5mm]
\max \{ C_5(\alpha,\mu), 2 (\mu_- c_0(d,\alpha)\delta_0(\alpha,\mu)^\alpha)^{1-2/\alpha} \},
&\hbox{if }1 < \alpha < 2.
\end{array}
\right.
$$
Due to the evident equality
\begin{equation*}
\label{A0_P}
\A^0(\bxi) P = \mu^0 V_\alpha (\bxi) P,
\end{equation*}
we have
\begin{equation*}
\label{res0_P}
(\A^0(\bxi) + \eps^\alpha I)^{-1}P = \left( \mu^0 V_\alpha(\bxi) + \eps^\alpha \right)^{-1} P.
\end{equation*}
This allows us to rewrite  \eqref{e2.37_all} as follows:
\begin{equation}
\label{e2.37_all_other}
\left\| (\A(\bxi)+\varepsilon^{\alpha}I)^{-1}- (\A^0(\bxi)+\varepsilon^{\alpha}I)^{-1}P
 \right\|\le
\check{C}_5(\alpha,\mu)\begin{cases} 1, &\hbox{if }  0 < \alpha <1,\\
(1 + | \operatorname{ln} \eps|)^2, &\hbox{if }  \alpha =1,
\\ \eps^{2 - 2\alpha}, &\hbox{if }  1< \alpha < 2,
\end{cases}
\end{equation}
for all $\varepsilon>0$ and $\bxi \in \wt{\Omega}$.
Using the discrete Fourier transform it can be shown that
\begin{equation*}
\label{discrete}
\left\| (\A^0(\bxi)+\varepsilon^{\alpha}I)^{-1}(I-P) \right\| = \sup_{0 \ne \n \in \Z^d}
(\mu^0 V_\alpha(2\pi \n+\bxi) + \eps^\alpha )^{-1}
\le (\mu_- c_0(d,\alpha) \pi^\alpha + \eps^\alpha)^{-1};
\end{equation*}
here we also used  \eqref{e1.16}  and the inequality $|2\pi \n+\bxi| \ge \pi$ which is valid for all $\bxi \in \wt{\Omega}$
and $0 \ne \n \in \Z^d$.
Therefore,
$$
\left\| (\A^0(\bxi)+\varepsilon^{\alpha} I)^{-1}(I-P) \right\| \le (\mu_- c_0(d,\alpha) \pi^\alpha)^{-1},
\ \ \varepsilon>0,\ \ \bxi \in \wt{\Omega}.
$$
Now the desired estimate \eqref{e2.37_2} follows from the latter inequality and \eqref{e2.37_all_other};
the constant $\mathrm{C}(\alpha,\mu)$ is equal to $\check{C}_5(\alpha,\mu) + (\mu_- c_0(d,\alpha) \pi^\alpha)^{-1}$,
if  $0< \alpha \le 1$, and equal to $\check{C}_5(\alpha,\mu) + (\mu_- c_0(d,\alpha) \pi^\alpha)^{1-2/\alpha}$, if $1< \alpha < 2$.
\end{proof}

\subsection{Approximation of the resolvent $(\A+\varepsilon^{\alpha}I)^{-1}$}

\begin{theorem}\label{teor3.6}
Let conditions \eqref{e1.1} and \eqref{e1.2} be fulfilled, and assume that   $0 < \alpha < 2$.
Then for all $\eps >0$  the following estimate holds:
\begin{equation}\label{th3.6_1}
\left\|(\A+\varepsilon^{\alpha}I)^{-1}- (\A^0+\varepsilon^{\alpha}I)^{-1}
\right\|_{L_2(\R^d) \to L_2(\R^d)}\le
{\mathrm C}(\alpha,\mu) \begin{cases} 1, &\hbox{if }  0 < \alpha <1,\\
(1 + | \operatorname{ln} \eps|)^2, &\hbox{if }  \alpha =1,
\\ \eps^{2 - 2\alpha}, &\hbox{if }  1< \alpha < 2;
\end{cases}
\end{equation}
the constant ${\mathrm C}(\alpha,\mu)$  depends only on $d$, $\alpha$, $\mu_-$ and $\mu_+$.
\end{theorem}

\begin{proof}
It follows from  \eqref{e1.15} and \eqref{direct_int} that  the operator under the norm sigh in  \eqref{th3.6_1}
admits the reporesentation
$$
|(\A+\varepsilon^{\alpha}I)^{-1}- (\A^0+\varepsilon^{\alpha}I)^{-1}
= {\mathcal G}^*\Bigl( \int_{\wt{\Omega}} \oplus \big((\A(\bxi)+\varepsilon^{\alpha}I)^{-1}- (\A^0(\bxi)+\varepsilon^{\alpha}I)^{-1}\big)\,d\bxi  \Bigr) {\mathcal G}.
$$
Therefore,
\begin{multline*}
\left\|(\A+\varepsilon^{\alpha}I)^{-1}- (\A^0+\varepsilon^{\alpha}I)^{-1}
\right\|_{L_2(\R^d) \to L_2(\R^d)}
\\
=
\sup_{\bxi \in \wt{\Omega}} \left\|(\A(\bxi)+\varepsilon^{\alpha}I)^{-1}- (\A^0(\bxi)+\varepsilon^{\alpha}I)^{-1}
\right\|_{L_2(\Omega) \to L_2(\Omega)},
\end{multline*}
and estimate \eqref{th3.6_1} follows from \eqref{e2.37_2}.
\end{proof}

\section{Homogenization of L\'evy-type operators}

\subsection{Main result}
Assuming that conditions \eqref{e1.1} and \eqref{e1.2} are satisfied, and that  $0< \alpha <2$,
we consider the closed quadratic form
\begin{equation*}
a_{\varepsilon}[u,u] = \frac{1}{2} \intop_{\R^d}  \intop_{\R^d} d\x \, d\y \, \mu^\eps(\x,\y)
\frac{|u(\x)-u(\y)|^2}{|\x - \y|^{d+\alpha}},\ \ u\in H^\gamma(\R^d),\ \ \varepsilon>0,
\end{equation*}
with $\mu^\eps(\x,\y):= \mu(\x/\varepsilon,\y/\varepsilon)$, and the family of self-adjoint operators $\A_\eps$ in  $L_{2}(\R^d)$ that are generated by this form.
We recall that the effective operator $\A^0$ and the effective coefficient $\mu^0$  are defined in \eqref{eff_op} and \eqref{2.33}, respectively.

With the help of scaling transformation, we deduce from Theorem  \ref{teor3.6} the following result:
\begin{theorem}
\label{teor3.1}
Let conditions \eqref{e1.1} and \eqref{e1.2} be fulfilled, and assume that  $0< \alpha <2$.
Then for all $\eps>0$ the following estimate holds:
\begin{equation}\label{e3.1}
\|(\A_{\varepsilon}+I)^{-1}-(\A^{0}+I)^{-1}\|_{L_2(\R^d) \to L_2(\R^d)}\le
{\mathrm C}(\alpha,\mu) \begin{cases} \eps^\alpha, &\hbox{if } 0 < \alpha <1,\\
\eps (1 + | \operatorname{ln} \eps|)^2, &\hbox{if } \alpha =1,
\\ \eps^{2 - \alpha}, &\hbox{if } 1< \alpha < 2.
\end{cases}
\end{equation}
The constant  ${\mathrm C}(\alpha,\mu)$ depends only on  $d$, $\alpha$, $\mu_-$ and $\mu_+$.
\end{theorem}
\begin{proof}
Consider the family of scaling transformations
\begin{equation*}
T_{\varepsilon}u(\x):=\varepsilon^{d/2}u(\varepsilon \x),\ \
\x\in\R^d,\ \ u\in L_{2}(\R^d), \ \ \eps >0.
\end{equation*}
A simple calculation shows that
$$
\A_{\varepsilon} = \varepsilon^{-\alpha} T_{\varepsilon}^* \A T_{\varepsilon},\quad \eps >0.
$$
Therefore,
\begin{equation}\label{e3.2}
(\A_{\varepsilon}+I)^{-1}= T_{\varepsilon}^{*}\varepsilon^{\alpha}(\A+\varepsilon^{\alpha}I)^{-1}T_{\varepsilon},
\ \ \varepsilon>0.
\end{equation}
Similar relations for the effective operator read
$$
\A^0 = \varepsilon^{-\alpha} T_{\varepsilon}^* \A^0 T_{\varepsilon},\quad \eps >0,
$$
and
\begin{equation}\label{e3.2_eff}
(\A^0+I)^{-1}= T_{\varepsilon}^{*}\varepsilon^{\alpha}(\A^0+\varepsilon^{\alpha}I)^{-1}T_{\varepsilon},
\ \ \varepsilon>0.
\end{equation}
From  \eqref{e3.2} and \eqref{e3.2_eff} taking into account the unitarity of the operator $T_\eps$ we obtain
$$
\|(\A_{\varepsilon}+I)^{-1}-(\A^{0}+I)^{-1}\|_{L_2(\R^d) \to L_2(\R^d)} =
\eps^\alpha \|(\A + \varepsilon^\alpha I)^{-1}-(\A^{0}+\eps^\alpha I)^{-1}\|_{L_2(\R^d) \to L_2(\R^d)}.
$$
Due to this relation the required estimate  \eqref{e3.1} follows from Theorem \ref{teor3.6}.
\end{proof}

\subsection{Concluding remarks}

1. An estimate of order $O(\eps^\alpha)$ in  \eqref{e3.1} is the best estimate that can be achieved if the homogenization  process is interpreted as a threshold effect at the edge of the spectrum. This follows from estimate \eqref{e2.38} and relation \eqref{e3.2}.  Thus, in the case  $0 < \alpha <1$ we obtain the best possible  precision by taking only the leading term
of the approximation.  For $\alpha$ from the interval $[1,2)$ the precision in \eqref{e3.1} is not the best possible any more,
it gets worse as $\alpha$ increases.  In this case the approximation can be improved by using correctors.
The authors are going to address this problem in a separate work.

2. Analyzing the expressions for the constants in the estimates that were used in deriving \eqref{e3.1} one can easily check
that the constant  ${\mathrm C}(\alpha,\mu)$ in \eqref{e3.1} depends only on the parameters
$d$, $\alpha$, $\mu_-$ and $\mu_+$. However, this constant is defined by different expressions
depending on whether   $0< \alpha <1$, or $\alpha=1$, or $1< \alpha <2$.
It should also be  noted that ${\mathrm C}(\alpha,\mu)$  tends to infinity, as $\alpha$ approaches $1$, both from the left
and from the right, and as $\alpha$ approaches $2$.

3. Theorem \ref{teor3.1} remains valid, if we replace the periodic lattice $\Z^d$ with an arbitrary periodic lattice in $\R^d$.
In this case the constant in estimate \eqref{e3.1} depends not only on the above mentioned parameters but also
on the characteristics of the lattice.

\bigskip\noindent
{\bf Acknowlegements. }  An essential part of this work was carried out during the visit of A. Piatnitski and E. Zhizhina
to the Euler International Mathematical Institute in St. Petersburg. The visit was supported by the grant of the Ministry
of Science and Higher Education of RF, agreement 075-15-2022-289 of 06.04.2022.   The work of A. Piatnitski and 
E.Zhizhina was partially supported by the project ``Pure Mathematics in Norway'' and the UiT Aurora project MASCOT.

\end{document}